\documentclass[12pt, leqno]{amsart}
\usepackage{amsmath,amsthm,amssymb,mathrsfs,stmaryrd,color}
\usepackage[all]{xy}
\usepackage{url}
\usepackage{geometry}
\usepackage{mathtools}

\usepackage{enumitem}
\setlist[enumerate]{itemsep=1em}

\usepackage{setspace}

\usepackage{quiver}
\usepackage{url}
\usepackage{todonotes}
\usepackage[headsepline]{scrlayer-scrpage}
\clearpairofpagestyles
\setkomafont{pageheadfoot}{\footnotesize}
\automark[section]{section}
\ohead{\pagemark}
\ihead{\headmark}

\usepackage[utf8]{inputenc}
\usepackage[T1]{fontenc}

\usepackage{relsize}
\usepackage[bbgreekl]{mathbbol}
\usepackage{amsfonts}

\DeclareSymbolFontAlphabet{\mathbb}{AMSb}
\DeclareSymbolFontAlphabet{\mathbbl}{bbold}

\usepackage[colorlinks=true,hyperindex, linkcolor=magenta, pagebackref=false, citecolor=cyan,pdfpagelabels]{hyperref}
\usepackage[capitalize]{cleveref}

\newcommand{\cosimp}[3]{\xymatrix@1{#1 \ar@<.4ex>[r] \ar@<-.4ex>[r] & {\ }#2 \ar@<0.8ex>[r] \ar[r] \ar@<-.8ex>[r] & {\ } #3 \ar@<1.2ex>[r] \ar@<.4ex>[r] \ar@<-.4ex>[r] \ar@<-1.2ex>[r] & \cdots }}

\newcommand{\colim}{\mathop{\mathrm{colim}}}

\newcommand{\adjunction}[4]{\xymatrix@1{#1{\ } \ar@<0.3ex>[r]^{ {\scriptstyle #2}} & {\ } #3 \ar@<0.3ex>[l]^{ {\scriptstyle #4}}}}

\makeatletter

\renewcommand{\tocsection}[3]{%
	\indentlabel{\@ifnotempty{#2}{\bfseries\ignorespaces#1 #2\quad}}\bfseries#3}
\renewcommand{\tocsubsection}[3]{%
	\indentlabel{\@ifnotempty{#2}{\ignorespaces#1 #2\quad}}#3}

\newcommand\@dotsep{4.5}
\def\@tocline#1#2#3#4#5#6#7{\relax
	\ifnum #1>\c@tocdepth
	\else
	\par \addpenalty\@secpenalty\addvspace{8pt}
	\begingroup \hyphenpenalty\@M
	\@ifempty{#4}{%
		\@tempdima\csname r@tocindent\number#1\endcsname\relax
	}{%
		\@tempdima#4\relax
	}%
	\parindent\z@ \leftskip#3\relax \advance\leftskip\@tempdima\relax
	\rightskip\@pnumwidth plus1em \parfillskip-\@pnumwidth
	#5\leavevmode\hskip-\@tempdima{#6}\nobreak
	\leaders\hbox{$\m@th\mkern \@dotsep mu\hbox{.}\mkern \@dotsep mu$}\hfill
	\nobreak
	\hbox to\@pnumwidth{\@tocpagenum{\ifnum#1=1\bfseries\fi#7}}\par
	\nobreak
	\endgroup
	\fi}

\AtBeginDocument{%
	\expandafter\renewcommand\csname r@tocindent0\endcsname{0pt}
}
\def\l@section{\@tocline{1}{8pt}{0pc}{5pc}{}}
\def\l@subsection{\@tocline{2}{6pt}{2.5pc}{5pc}{}}
\def\l@subsubsection{\@tocline{3}{4pt}{5pc}{5pc}{}}

\makeatother

\title{On relative vertical compactification of \\ weakly square complete adic spaces}
\author{Ronald Solodov}
\date{}

\begin{document}

\newtheorem{theorem}{Theorem}[section]
\newtheorem*{theorem*}{Theorem}
\newtheorem*{definition*}{Definition}
\newtheorem{proposition}[theorem]{Proposition}
\newtheorem{lemma}[theorem]{Lemma}
\newtheorem{corollary}[theorem]{Corollary}

\theoremstyle{definition}
\newtheorem{definition}[theorem]{Definition}
\newtheorem{question}[theorem]{Question}
\newtheorem{remark}[theorem]{Remark}
\newtheorem{warning}[theorem]{Warning}
\newtheorem{example}[theorem]{Example}
\newtheorem{notation}[theorem]{Notation}
\newtheorem{convention}[theorem]{Convention}
\newtheorem{construction}[theorem]{Construction}
\newtheorem{claim}[theorem]{Claim}
\newtheorem{assumption}[theorem]{Assumption}

\crefname{assumption}{assumption}{assumptions}
\crefname{construction}{construction}{constructions}

\begin{abstract}
	
	We prove for a morphism $f \colon X \rightarrow S$ locally of $^+$weakly finite type, separated and taut, where $X$ is a weakly square complete adic space and $S$ a square complete and stable adic space, there exists a universal vertical compactification $f' \colon X' \rightarrow S$. This provides a generalized version of Huber's proof of universal compactification of morphisms between analytic adic spaces. Notably, we will see that for the compactification, it is not necessary to assume that the adic spaces are locally Noetherian. In the fourth section we give an explicit and simple construction of $X'$. 
	
\end{abstract}

	\maketitle

	\tableofcontents
	\newpage 	
	\thispagestyle{empty}
	
\section{Introduction}
\subsection{History and statement of main theorem}

The relative compactification of schemes was already proven by  M. Nagata in \cite{Nag62} and \cite{Nag63}. More precisely, it was shown that a separated and finite type morphism to a Noetherian scheme $S$ can be factored into an open immersion followed by a proper morphism. In \cite{Hu1996}, Huber proved that a separated, locally of $^+$weakly finite type and taut morphism $X \rightarrow S$ of analytic adic spaces can be factored into an open immersion $X \hookrightarrow X'$ followed by a partially proper morphism. In contrast to M. Nagata's result, the space $X'$ is unique up to unique isomorphism. In this article, we generalize Huber’s result to some extent, and we expect that our construction will contribute to extending the six functor formalism to a wider class of compactifiable morphisms. We now state our main theorem
\bigskip
\begin{theorem}[Theorem \ref{main}] \label{main'}
	
	Let $f \colon X \rightarrow S$ be a locally of $^+$weakly finite type, separated and taut morphism of adic spaces, where $X$ is weakly square complete and $S$ a square complete and stable adic space. Then there exists a universal vertical compactification $f' \colon X' \rightarrow S$ of $f$.
	
\end{theorem}
It will turn out in Proposition \ref{weaker} that $X'$ is necessarily weakly square complete if $X' \rightarrow S$ is a universal vertical compactification of a morphism $X \rightarrow S$ with $S$ a square complete adic space. 

By a square complete adic spaces $S$, we refer to the following property satisfied by every affinoid space: namely, that for any specialization diagram with points in $S$ of the given form
\begin{center}
	\begin{tikzcd}
		x_1 \arrow[r, leftsquigarrow] & x \\
		 & x_2 \arrow[u, leftsquigarrow] \\
	\end{tikzcd}
\end{center}
there exists a unique point $x'$ in $S$ completing the diagram in the lower left corner. With $y \rightsquigarrow x$ we mean that $x$ is a specialization of $y$ and if this arrow occurs in a diagram its orientation means horizontal specialization or vertical specialization. If $x'$ is not unique then we call the adic space weakly square complete.
\bigskip

In the course of proving the theorem, we will see that Huber’s original assumption in \cite[~ 1.1.1]{Hu1993}  that all adic spaces are locally Noetherian is not required. Instead, it is sufficient to assume that they are stable adic. Those assumptions appear solely in \cite[Lemma1.3.14(i)]{Hu1996}, and we will treat them in Lemma\ref{mod}.

Another main goal in this paper is to give a simple construction of $X'$. For this we will construct for every morphism $f \colon X \rightarrow S$ of adic spaces the set $\mathrm{Spa}(X,S)$, which consists of triples $(x,V,s)$ where $x \in X$ has no vertical generalizations, $(x,V)$ is a valuation ring on $X$ and $s$ is a center of $(x,V)$ on $S$. The set $\mathrm{Spa}(X,S)$ will be endowed with a topology and we will see:
\bigskip
\begin{theorem}[Theorem \ref{spa}]
	
	Let $f$ be like in Theorem \ref{main'} and $f' \colon X' \rightarrow S$ the universal vertical compactification of $f$. Then $\mathrm{Spa}(X,S)$ is homeomorphic to $X'$.
	
\end{theorem}

\subsection{Sketch of proof}

We give a brief summary about the original proof: Huber's proof in \cite[Theorem ~ 5.1.5]{Hu1996} crucially depends on the notion of valuation rings on analytic adic spaces. 

\begin{definition}\label{special}
			Let $X$ be an analytic adic space.
	\begin{enumerate}[label=\roman*)]	

		\item  A $valuation \ ring$ $of$ $X$ is a pair $(x,A)$, where $x \in X$ and $A$ is a valuation ring of the residue class field $k(x)$ of $X$ with $A \subset k(x)^+$ (sometimes we write only $A$ instead of $(x,A)$). The point $x$ is called the $support$ of $(x,A)$. 
		\item A point $y \in X$ is called a $center$ $of$ $(x,A)$ if $y$ is a specialization of $x$ with $k(y)^+ = i^{-1}(A)$ where $i \colon k(y) \rightarrow k(x)$ is the natural ring homomorphism.	
		\item For an analytic adic space $X$, we denote by $X_{\nu}$ the set of all valuation rings of $X$.
	\end{enumerate}
\end{definition}  

As a first step, Huber shows it is enough to consider an open affinoid covering $(S_i)_i$ of $S$ and construct for each morphism $f^{-1}(S_i) \rightarrow S_i$ a universal compactification $X'_i \rightarrow S_i$. \cite[Lemma ~ 5.1.7]{Hu1996} shows that the morphism $X' \rightarrow S$ which arise from gluing all the universal compactification $X'_i \rightarrow S_i$ is a universal compactification of $X \rightarrow S$. Hence, for the proof we can assume that $S$ is an affinoid space.   Then for every open affinoid subset $U$ of $X$  a certain open subset $U_d$ in $U_c$ (or more precisely $\overline{U}^S$) is brought into consideration. Note that $U_c$ is the universal compactification of $U \rightarrow S$ for $U$ and $S$ affinoid adic spaces. The subset $U_d$ consists of centers of certain valuation rings of $U$. The rest of the proof consists in gluing the spaces $U_d$ corresponding to each open affinoid subset $U \subset X$, and in showing that the resulting morphism $X' \rightarrow S$ defines a universal compactification of $X \rightarrow S$, once again invoking \cite[Lemma~5.1.7]{Hu1996}. We will adapt Huber's proof and for this we will assume that $X$ is weakly square complete and $S$ a square complete and stable adic space. The latter is invoked solely to avoid the need for assuming that the adic spaces are locally Noetherian.

\subsection{Overview of the sections}

\begin{itemize}
	\item Section \ref{2}: We recall some properties of valuated valuation rings. The definition of the center of a valuated valuation ring (Definition \ref{vvr}) is important for constructing a universal vertical compactification $f' \colon X' \rightarrow S$, since we will glue $X'$ together from these centers.
	\item Section \ref{3}: In this main section, we define the universal vertical compactification and prove its existence by following Huber’s original argument.	\begin{itemize}
				\item First we recollect some properties of morphism between adic spaces. Then in Definition \ref{proposition} we define vertically partially proper morphism and in Definition \ref{comp} universal vertical compactification.
				\item Here we proof the existence of a universal vertical compactification $f' \colon U_c \rightarrow S$ for certain morphism $U \rightarrow S$ between affinoid adic spaces. In this case $U_c$ will again be affinoid and we can give a concrete description of it. 
				\item Finally, we show for which morphism $f \colon X \rightarrow S$ a universal vertical compactification $f' \colon X' \rightarrow S$ exist, where $X$ is weakly square complete and $S$ a square complete stable adic space. In particular, $X'$ will be glued together of centers in $U_c$ of certain valuated valuation rings of $U$ for every open affinoid subspace $U$ of $X$. Notably, we can avoid locally Noetherian adic spaces. The hypothesis is only needed in \cite[Lemma ~ 1.3.14 i)]{Hu1996}, which is required to show that the datum defined in Theorem \ref{main} is indeed a gluing datum. In Lemma \ref{mod} we provide an alternative argument that the proof of \cite[Lemma ~ 1.3.14 i)]{Hu1996} remains valid, adding a natural condition and removing the Noetherian hypothesis.
		\end{itemize}
	\item  Section \ref{4}: For every morphism $f \colon X \rightarrow S$ of adic spaces we construct a set $\mathrm{Spa}(X,S)$ which is in bijection with $X'$ if the universal vertical compactification $f' \colon X' \rightarrow S$ of $f$ exists. In particular, we endow $\mathrm{Spa}(X,S)$ with the final topology induced by the family of maps $\mathrm{Spa}(U,V) \rightarrow \mathrm{Spa}(X,S)$ for every open affinoid subspace $U \subset X$ and $V \subset S$. The topology on $\mathrm{Spa}(U,V)$ is defined by $\overline{U}^V$, the universal vertical compactification of $U \rightarrow V$, provided this morphism is of $^+$weakly finite type.
\end{itemize}

\subsection{Notes on the references}

	In this article, we will refer to some statements in \cite{Hu1993} and \cite{Hu1996}. We want to point out that in \cite[Definition ~ 3.6.4]{Hu1993} an affinoid adic space is an object which is isomorphic to a Spa$(A,A^+)$ where $A$ has a Noetherian ring of definition and an adic space has an open covering of such affinoid adic spaces. But all statements which we will use are also true for adic spaces in the usual sense without any assumption about the ring of definition of $A$. We would also like to note that in this source, valuations are defined additively, so the roles of maximum and minimum of points are reversed. However, we will already mention here that, although in \cite{Hu1996} Huber makes overall assumptions about all Huber rings he considers, these assumptions are not necessary for our purpose at all. Still, there will be instances in which we require, given a morphism $X \rightarrow S$ locally of weakly finite type, that $S$ is a stable adic space to ensure the existence of fiber products with itself.
	
\bigskip

$\mathbf{Acknowledgements.}$ The author would like to thank Katharina Hübner for taking the time to discuss various aspects of the thesis in detail; in particular, the author is grateful for the ideas and approaches to the proofs in the final chapter. Especially the statements in the final sections are based on ideas originally suggested by her.  The author also wishes to thank Torsten Wedhorn for suggesting the topic, thereby opening the door to the study of adic spaces, as well as for the general academic guidance throughout the thesis. Last but not least, the author is thankful to Roland Huber for answering emails and providing a better understanding of his original proof. Furthermore, Jon Miles is thanked for his helpful comments on how the article should be structured and formatted.

The author would like to express his heartfelt thanks to his fellow students Paul Burk, Michelle Klemt, Johann Gramzow, Hannah Laus, Niklas Ludwig, Aaron Rauchfuß and Benjamin Steklov, who accompanied him throughout his studies and with whom he had the pleasure of learning and exploring mathematics together.

\section{Background on valuated valuation rings} \label{2}

Since the specialization of two points in analytic adic spaces is always vertical, we have to understand the behaviour of valuation rings in general adic spaces. Fortunately, Huber already defined a more general version of valuation rings on adic spaces in \cite[Definition ~ 3.11.8]{Hu1993} and with it he has established several results:

\begin{definition}\label{vvr}
	\begin{enumerate}[label=\roman*)]
		
		\item A $valuated$ $local$ $ring$ is a local ring $A$ with a valuation $\nu$ on $A$ such that the support of $\nu$ is the maximal ideal of $A$. Let $A_{\nu}$ be the subring of those elements of $A$ with values smaller or equal one. We sometimes write more explicitly $(A,A_{\nu})$ if $A$ is a valuated local ring.
		
		\item Let $A$ and $B$ be valuated local rings with valuations $\nu$ and $\omega$. A \emph{ring homomorphism} $f \colon A \rightarrow B$ between valuated local rings is a ring homomorphism which is local, $f(A_{\nu}) \subset B_{\omega}$ and the induced ring homomorphism $A_{\nu} \rightarrow B_{\omega}$ is local.

		\item	A $valuated \ valuation \ ring$ is a valuated local ring $A$ which is a valuation ring. 
		
		\item Let $X$ be an adic space and $x \in X$. We call $(x,A, A_{\nu})$ a $valuated$ $valuation$ $ring$ of $X$ with support $x$ if  $(A,A_{\nu})$ is a valuated valuation ring such that $A \subset k(x)$, $\text{Quot}(A) = k(x)$ and $A_{\nu} \subset k(x)^+.$
		
		\item We call a valuated valuation ring of an adic space $X$ of the form $(x,k(x), k(x)_{\nu})$ $valuation$ $ring$ and write shortly $(x,A)$ for a valuation ring $A \subset k(x)^+$.
		
		\item Let $X$ be an adic space, $A$ a valuated valuation ring of $X$ with support $x$. A $center$ of $A$ is a specialization $y$ of $x$ such that for the canonical map $i\colon \mathcal{O}_{X,y} \rightarrow k(x)$ we have $i(\mathcal{O}_{X,y}) \subset A$ and its restriction $\mathcal{O}_{X,y} \rightarrow A$ is a local ring homomorphism between the valuated local rings $(\mathcal{O}_{X,y}, \nu_y)$ and $(A, \nu)$.
	\end{enumerate}
\end{definition}

\begin{remark}
	\begin{enumerate}[label=\roman*)]
		\item If $A$ is a valuated valuation ring, then $A_{\nu}$ is a valuation ring with the same field of fractions as $A$.
		
		\item If $(x,A,A_{\nu})$ is a valuated valuation ring of an adic space $X$, then $(x, A_{\nu})$ is a valuation ring.
		
		\item For valuation rings $(x,A)$ the definition of a center simplifies to Definition \ref{special} ii). 
	\end{enumerate}
\end{remark}

We collect and prove some facts about valuated valuation rings.

\begin{lemma}\label{center}
	Let $X$ be an adic space.
	\begin{enumerate}[label=\roman*)]
		\item Let $A$ be a valuation ring on $X$ with support $x$ and $y$ a center of $A$. Then $y$ is a vertical specialization of $x$.
		\item Let $x,y \in X$ and let $y$ be a vertical specialization of $x$. Then there exists a  valuation ring $(x,A)$ with support $x$ and center $y$.
		\item Let $y$ be a vertical specialization of $x$, then there exists exactly one valuated valuation ring on $X$ with support $x$ and center $y$.
		\item For every $x,y \in X$ where $y$ is a specialization of $x$, exists a valuated valuation ring $(x,A,A_{\nu})$ with center $y$.
	\end{enumerate}
\end{lemma}

\begin{proof}
See \cite[Lemma~3.11.10]{Hu1993}
\end{proof}

The following lemma will be used frequently. 

\begin{lemma}\label{criterion}
		Let $X$ be an affinoid adic space and $A$ a valuated valuation ring of $X$ with support $x$.
	\begin{enumerate}[label=\roman*)]

		\item Then $A$ has a center on $X$ if and only if for the canonical map $i \colon \mathcal{O}_X(X) \rightarrow k(x)$ the inclusions
		$i(\mathcal{O}_X(X)) \subset A$ and $i(\mathcal{O}_X^+(X)) \subset A_{\nu}$ hold.
		
		\item Every valuated valuation ring on $X$ has at most one center on $X$.
		
		\item Let $z$ be a center of $(x,A,A_{\nu})$ and $z'$ a center of $(x,A',A'_{\nu})$ with $A \subset A'$ and $A_{\nu} \subset A'_{\nu}$, then $z$ is a specialization of $z'$.
		
		\item Let $(x,A,A_{\nu})$ be a valuated valuation ring with center $z$. Then $(x,A_{\nu})$ has a center $h$ and $z$ is a horizontal specialization of $h$.
		
	\end{enumerate}
\end{lemma}

\begin{proof}
For i) and ii) see \cite[Lemma~3.11.9]{Hu1993}.

To prove iii) assume $z$ is not a specialization of $z'$, then there exists an open affinoid neighbourhood $U$ of $z$ not containing $z'$. By i) $(x,A',A'_{\nu})$ has a center in $U$ unequal to $z'$, because $(x,A,A_{\nu})$ has a center in $U$. But this is a contradiction to ii).

Finally for iv) we assumed $(x,A,A_{\nu})$ has a center on $X$. We get by i) that $(x,A_{\nu})$ has a center $h$ on $X$. Since the valuation ring $(x,A_{\nu})$ is the same as a valuated valuation ring $(x,k(x), A_{\nu})$, we get by iii) that $z$ is a specialization of $h$. By \cite[Corollary ~ 3.11.7]{Hu1993} we have to show that $\mathcal{O}^+_{X,z} \rightarrow \mathcal{O}^+_{X,h}$ is local, then $z$ is a horizontal specialization of $h$. We observe that the composition
\begin{center}
	\begin{tikzcd}
		\mathcal{O}^+_{X,z} \arrow[r]  & \mathcal{O}^+_{X,h} \arrow[r] & A_{\nu}
	\end{tikzcd}
\end{center}
is a local ring homomorphism. Hence the first arrow is a local ring homomorphism.

\end{proof}

\begin{lemma} \label{val}
	Let $X$ be affinoid adic space. The following statements hold:
	\begin{enumerate}[label=\roman*)]
		\item If $(x,B), (x,B')$ are valuation rings of $X$ with centers $z,z'$ on $X$ and $B \subset B'$, then $z$ is a specialization (already vertical) of $z'$.
		\item If $z \in X$ is a center of $(x,B)$ and $A \subset k(x)$ valuation ring with $k(z) \cap B \subset A$, then $B \subset A.$
		\item If $z \in X$ is a center of $(x,A), (x',A')$ and $x$ is a vertical specialization of $x'$, then $A' \cap k(x) = A$.  
	\end{enumerate}
\end{lemma}

\begin{proof}

For i) we have that $z$ is a specialization of $z'$ by Lemma \ref{criterion} iii). Since $z$ and $z'$ are centers of valuation rings with support $x$, they are vertical specializations of $x$ by Lemma \ref{center} i). Hence, they have the same support as $x$ and so $z$ is a vertical specialization of $z'$. 
\\
Let us prove ii). We recall the following general topological statement: If $Y \subset X$ is a dense subset and $U \subset X$ is an open subset, then $\overline{Y \cap U} = \overline{U}$. Together with Remark \ref{spec. field} i), which says that $k(z)$ is dense in $k(x)$, and since valuation rings are open and closed subsets, this implies ii).
\\
We assume the situation from iii). By \cite[1.1.14 e)]{Hu1996} $A' \cap k(x)$ is a valuation ring of $k(x)$. Since $z$ is a center of $(x',A')$, this implies $z$ is also a center of $(x, A' \cap k(x))$ by Lemma \ref{center}. Also i) of the same lemma implies that there is a unique valuation ring with support $x$ which has $z$ as center. Hence, $A' \cap k(x) = A.$

\end{proof}

In the following, we show that valuated valuation rings of $X$ can be mapped to a valuated valuation ring of $Y$.

\begin{proposition}\label{cut}
	Every morphism of adic spaces $f \colon X \rightarrow Y$ sends a valuated valuation ring $(x,A,A_{\nu})$ of $X$ to the valuated valuation ring $(f(x), A \cap k(f(x)), A_{\nu} \cap k(f(x))$. Furthermore, if $p$ is a center of $(x,A,A_{\nu})$ then $f(p)$ is a center of \\ $(f(x), A \cap k(f(x)), A_{\nu} \cap k(f(x))$.
\end{proposition}

\begin{proof}
It follows by definition that $(f(x), A \cap k(f(x)), A_{\nu} \cap k(f(x)))$ is a valuated valuation ring on $Y$ and the induced ring homomorphism $\mathcal{O}_{Y,f(p)} \rightarrow A\cap k(f(x))$ is a morphism of the valuated local rings $(\mathcal{O}_{Y,f(p)}, \nu_{f(p)})$ and $(A \cap k(f(x)), \nu \circ \iota)$, where $\iota \colon k(f(x)) \rightarrow k(x)$ is the canonical morphism.

\end{proof}

\begin{remark} In the presence of a morphism $f \colon X \rightarrow Y$ of adic spaces,
	we call $y \in Y$ a center of a valuated valuation ring $(x,A, A_{\nu})$ of $X$ if $y$ is a center of $(f(x), A \cap k(f(x)), A_{\nu} \cap k(f(x)))$ and also refer to the last valuated valuation ring as valuated valuation ring $(x,A)$ on $Y$ if there is no ambiguity about which map is used. 
\end{remark}

\begin{definition}
	Let $X$ be an adic space and $x \in X$. Then we call the residue field of $\mathcal{O}^+_{X,x}$ the $specialization \ field$ $k(x)^{\succ}$ of $x$.
\end{definition}

\begin{remark}\label{spec. field}
	\begin{enumerate}[label=\roman*)]
		
		\item If $x$ is a vertical specialization of $x'$ in an adic space , then the image of $k(x) \rightarrow k(x')$ is dense. This follows by \cite[Proposition ~ 12.7 (1)]{Iv25} which says, that the completions of the valued field of a point $x$ and the residue field of $x$ are isomorphic. Since the valued fields of $x$ and $x'$ are the same in our case, this shows our remark. 
		
		\item 	If $x$ is a horizontal specialization of $x'$ in an adic space $X$, then by \cite[Corollary ~ 3.11.7 i)]{Hu1993} the canonical morphism $\mathcal{O}^+_{X,x} \rightarrow \mathcal{O}^+_{X,x'}$ is local. Hence we get a morphism between the specialization fields $k(x)^{\succ} \rightarrow k(x')^{\succ}.$ Note that the residue fields of $\mathcal{O}^+_{X,x}$ and $k(x)^+$ are canonically isomorphic, see \cite[Proposition ~ 12.7 (2)]{Iv25}
		
		\item By \cite[Remark ~ 2.5]{Wed19} we have a bijection between the valuation rings in $k(x)^{\succ}$ and the valuation rings of $k(x)$ which are contained in $k(x)^+$.
		
		\item If $x$ is a non-analytic point in an adic space, then the residue field $k(x)$ carries the discrete topology. Hence, every valuation ring of $k(x)$ is clopen. On the other hand, if $x$ is an analytic point, then $k(x)$ is endowed with a topology induced by a rank 1 valuation. In particular, by \cite[Proposition 5.45]{Wed19}, for every non-trivial valuation on $k(x)$, its valuation ring, which either contains or is contained in $k(x)^+$, is clopen.

	\end{enumerate}
\end{remark}

The following proposition and definition are not found in the literature. Although they are clear, we will state them for the sake of completeness.

\begin{proposition}
	Let $U \coloneq$ $\mathrm{Spa}(A,A^+)$ be an affinoid space and $x,y \in \mathrm{Spa}(A,A^+)$ with $x$ a specialization of $y$. Consider the canonical morphism $g \colon \mathcal{O}_{U,x} \rightarrow \mathcal{O}_{U,y}$. Then $x$ is a generalized horizontal specialization of $y$ if and only if $v_x$ is a generalized horizontal specialization of $\mathrm{Spv}(g)(v_y)$.
\end{proposition}

\begin{proof}

Both directions are  well known if we assume $x$ is a horizontal specialization of $y$. So it is left to show the other case in the definition of generalized horizontal specialization. First we note that if a trivial valuation is a horizontal specialization of some valuation then the characteristic group of this valuation is trivial. We assume first that the characteristic group $c\Gamma$ of  $\mathrm{Spv}(g)(v_y)$ is trivial and $v_x$ is a trivial valuation with $\mathrm{supp(Spv}(g)(v_y)_{| c\Gamma}) \subset \mathrm{supp}(v_x$). Because the following diagram is commutative
\begin{center}
	\begin{tikzcd}
		\mathcal{O}_{U,x} \arrow[rr, "g"] & & \mathcal{O}_{U,y} \\
		& A, \arrow[ul, "\phi"]  \arrow[ur, "\psi", swap] & \\
	\end{tikzcd}
\end{center}  

we get that the characteristic group $c\Gamma_y$ of $y$ is also trivial. Also by the commutativity of the diagram it shows supp$(y_{|c\Gamma_y}) \subset$ supp$(x)$.

Now we assume $x$ is a trivial valuation and the characteristic group $c\Gamma_y$ is trivial of $y$ with supp$(y_{|{c\Gamma_y}})  \subset $ supp$(x)$. Since $x$ is a trivial valuation, we find that $v_x$ is one too. The same holds for $v_{y_{|{c\Gamma_y}}}$. We consider again a commutative diagram:

\begin{center}
	\begin{tikzcd}
		
		\mathcal{O}_{U,x} \arrow[r, "f"] \arrow[rr, "g", bend left] & \mathcal{O}_{U,y_{|c\Gamma_y}} \arrow[r, "h"] & \mathcal{O}_{U,y}.  
		
	\end{tikzcd}
\end{center}

We know that $(y_{|{c\Gamma_y}})$ is a horizontal specialization of $y$. Hence, $v_{(y_{|{c\Gamma_y}})}$ is a horizontal specialization of Spv$(h)(v_y)$. It follows that Spv$(f)(v_{(y_{|{c\Gamma_y}})})$ is a trivial valuation and a horizontal specialization of Spv$(g)(v_y)$. This implies that the characterstic group of Spv$(g)(v_y)$ is trivial. It is clear that every support of a valuation on $\mathcal{O}_{U,x}$ is contained in the support of $v_x.$ 

\end{proof}
As for the horizontal specializations and vertical specializations, the last proposition justifies the following definition.

\begin{definition}
	Let $X$ be an adic space and $x,y \in X$. We call $x$ a $generalized$  $horizontal$  $specialization$ of $y$ if $v_x$ is a generalized horizontal specialization of Spv$(g)(v_y)$, where $g \colon \mathcal{O}_{X,x} \rightarrow \mathcal{O}_{X,y}$ is the canonical morphism.
\end{definition}

\section{Universal vertical compactification} \label{3}

In this section we will recall some definitions and statements, define the notion of a vertically partially proper morphism in Definition \ref{proposition} and in particular, in Definition \ref{comp} we define the notion of a universal vertical compactification. We will, in Theorem \ref{baby}, prove an affinoid version of the universal vertical compactification and Theorem \ref{main} will show a generalized version. At first sight, it appeared that the original proof of \cite[Theorem~5.1.5]{Hu1993}, which we aim to generalize, relied only on the property that every specialization between two points is already a vertical specialization in an analytic adic space. However, more is needed to show the openness of a certain subset in Theorem \ref{main}, which is why we define square complete adic spaces in Definition \ref{weaker} and along with it, we require Lemma \ref{horizontal}. Also we aim to show that it is enough to consider stable adic spaces. For this the Corollary \ref{iso} will be needed.

\subsection{Background and main definition}

\begin{definition}
	Let $f \colon X \rightarrow Y$ be a morphism of adic spaces. 
	\begin{enumerate}[label=\roman*)]
		
		\item $f$ is called $adic$ if for every $x$ there exist open affinoid subspace $U,V$ of $X,Y$ such that $x \in U$, $f(U) \subset V$ and $\mathcal{O}_Y(V) \rightarrow \mathcal{O}_X(U)$ is adic.\footnote{Then for every pair $U,V$ of open affinoid subspaces of $X,Y$ with $f(U) \subset V, \mathcal{O}_Y(V) \rightarrow \mathcal{O}_X(U)$ is adic.}
		
		\item $f$ is called \emph{locally of weakly finite type} if, for every $x \in X$, there exist open affinoid subspaces $U,V$ of $X,Y$ such that $x \in U, f(U) \subset V$ and the ring homomorphism of Huber rings $\mathcal{O}_Y(V) \rightarrow \mathcal{O}_X(U)$ is of topologically finite type.
		
		\item $f$ is called \emph{locally of $^+$weakly finite  type} if, for every $x \in X$, there exist open affinoid subspaces $U,V$ of $X,Y$ and a finite subset $E$ of $\mathcal{O}_X(U)$ with $x \in U, f(U) \subset V$, the ring homomorphism of Huber rings  $\mathcal{O}_Y(V) \rightarrow \mathcal{O}_X(U)$ is of topologically finite type and $\mathcal{O}_X^+(U)$ is the smallest ring of integral elements of $\mathcal{O}_X(U)$ which contains $E$ and the image of $\mathcal{O}_Y^+(V)$ in $\mathcal{O}_X(U)$, that means $\mathcal{O}_X^+(U)$ is the integral closure of $\mathcal{O}_Y^+(V)[E \cup \mathcal{O}_X(U)^{\circ \circ}]$ in $\mathcal{O}_X(U)$.
		
		\item $f$ is called \emph{locally of finite type} if, for every $x \in X$, there exist open affinoid subspaces $U,V$ of $X,Y$ such that $x \in U, f(U) \subset V$ and the ring homomorphism of Huber pairs $(\mathcal{O}_Y(V)), \mathcal{O}_Y^+(V)) \rightarrow (\mathcal{O}_X(U), \mathcal{O}_X^+(U))$ is of topologically finite type.
	\end{enumerate}
\end{definition}

In general the corresponding morphism $\mathrm{Spa} B \rightarrow \mathrm{Spa} A$ to a ring homomorphism $A \rightarrow B$ between Huber pairs is not spectral. It is the case if $A \rightarrow B$ is an adic morphism of Huber rings. We recall some statements which will be used. 

\begin{proposition}
	\begin{enumerate}[label=\roman*)]
		\item In the definition above we have that $iv) \implies iii) \implies ii) \implies i)$.
		
		\item Let $A$ be a Huber ring and $v$ a valuation on $A$. Then the induced morphism $A \rightarrow K(v)$ to the valued field of $v$ is adic.
		
		\item Let $f \colon X \rightarrow Y$ be an open embedding of adic spaces. Then $f$ is adic.
		
		\item Let $f \colon X \rightarrow Y$ be an adic quasi-compact morphism of adic spaces. Then $f(X)$ is pro-constructible.
		
		\item Let $f \colon X \rightarrow Y$ be an adic morphism of adic spaces. We assume that a valuation ring $(x,A)$ has a center on $Y$. Then the induced morphism $\mathrm{Spa}(k(x),A) \rightarrow Y$ is adic.
		
		\item Let $X$ be an adic space and $x$ a point of $X$. Then $\mathrm{Spa}(k(x), k(x)^+)$ is a stable adic space.
	\end{enumerate}
\end{proposition}

\begin{proof}

The first four statements are well known and not the focus of this work. Nevertheless, we briefly justify the last two statements:
By ii) we know that $\mathrm{Spa}(k(x),k(x)^+) \rightarrow X$ is an adic morphism. Since $f$ is adic we get the resulting composition $\mathrm{Spa}(k(x),k(x)^+) \rightarrow Y$ is adic. Hence there is some affinoid neighbourhood $\mathrm{Spa}(A,A^+)$ in $Y$ such that $\hat{A} \rightarrow \widehat{k(x)}$ is an adic morphism. This shows v). \\
For vi) we distinguish the following cases: If $x$ is a non-analytic point then $k(x)$ has the discrete topology and thus it is stably sheafy. In the other case if $x$ is an analytic point then $k(x)$ has a topology induced by a height $1$ valuation.  Hence, by \cite[Remark ~ 6.37 (2)]{Wed19} $k(x)$ is strongly Noetherian and this implies also being stably sheafy.

\end{proof}

We recall that for a morphism $X \rightarrow S$ locally of weakly finite type and an adic morphism $Y \rightarrow S$, where $Y$ is a stable adic space, the fiber product $X \times_S Y$ exists in the category of adic spaces, see \cite[Theorem ~ 8.56 (b)]{Wed19}

\begin{definition} \label{proposition}  Let $f \colon X \rightarrow Y$ be a morphism of adic spaces, where $X$ is a stable adic space.
	\begin{enumerate}[label=\roman*)] 
		\item We call $f$ $vertically \ specializing \ at \ a \ point  \ x \in X$ if for every vertical specialization $y'$ of $f(x)$ there exists a specialization $x'$ of $x$ with $y' = f(x')$.  We further call $f$ $universally \ vertically \ specializing \ at \ a \ point \ x \in X$ if $f$ is locally of weakly finite type and for every adic morphism of adic spaces $Y' \rightarrow Y$ with $Y'$ a stable adic space and every point $x'$ of $X \times_Y Y'$ lying over $x$ the projection $X \times_Y Y' \rightarrow Y'$ is vertically specializing at $x'$. The morphism $f$ is called $vertically \ specializing$ respectively $universally \ vertically \ specializing$ if it is so at every point in $X$.
		\item We call $f$ \emph{vertically separated} if $f$ is quasi-separated, locally of weakly finite type and for any valuation ring $(x,A)$ with centers $u,v$ and $f(u)=f(v)$ it follows that $v=u$. 
		
		\item A morphism of adic spaces $f \colon X \rightarrow Y$, where $X$ is a stable adic space, is called \emph{separated} if $f$ is locally of weakly finite type and the image of the diagonal morphism $\Delta \colon X \rightarrow X \times_Y X$ is closed.
		
		\item We call $f$ \emph{vertically partially proper} if $f$ is locally of $^+$weakly finite type, separated and universally vertically specializing. 
	\end{enumerate}
\end{definition}

We show that for a morphism universally vertically specializing at $x$ one can lift vertical specializations to vertical specializations.

\begin{proposition}
	
	Let $f$ be universally vertically specializing at $x$. Then for every vertical specialization $y$ of $f(x)$, there exists a vertical specialization $z$ of $x$ with $y = f(z)$.  
	
\end{proposition}

\begin{proof}
The proof is essentially based on the implication $a) \implies b)$ of \cite[Proposition ~ 1.3.8]{Hu1996} from which we extract only the necessary information.
Let $A$ be a valuation ring of $x$ such that $y$ is a center of $(x,A)$ on $Y$. We consider the cartesian diagram

\begin{center}
	\begin{tikzcd}
		X' \arrow[r, "g'"] \arrow[d, "f'"] & X \arrow[d, "f"] \\
		\mathrm{Spa}(k(x),A) \arrow[r, "g"] & Y.\\
	\end{tikzcd}
\end{center}

Let $y' \in \mathrm{Spa}(k(x),A)$ be the closed point and $x''$ a vertical generalizations of $y'$ with $g(x'') = f(x).$ Such a point $x''$ exists because $g \colon \mathrm{Spa}(k(x,A)) \rightarrow Y$ is adic and by \cite[Proposition 3.8.9]{Hu1993} the map $g$ is surjective between all vertical generalizations of $y'$ and $g(y') = y$.  Then by \cite[Lemma ~ 8.58]{Wed19} there exists a point $x' \in X'$ with $g'(x') = x$ and $f'(x')= x''$. Since $f'$ is vertically specializing at the point $x'$ we get a point $z' \in X'$ such that $z'$ is a specialization of $x'$ with $f'(z') = y'$. In particular we get that $z \coloneqq g'(z')$ is a specialization of $x$. We show that $z$ is even a vertical specialization of $x$. Since $f(z) = y$ this will finish the proof. Note that we have $\widehat{k(x)} = k(y')$. Hence, we get the following commutative diagram 
\begin{center}
	\begin{tikzcd}
		\widehat{k(x)} \arrow[rrrr, "f'^{\flat}_{z'}"] & & & & k(z') \\
		& k(x) \arrow[ul] & & k(z) \arrow[ur, "g'^{\flat}_{z'}"] & \\
		& & \mathcal{O}_{X,z}, \arrow[ur] \arrow[ul] & & \\
	\end{tikzcd}
\end{center}
where $\mathcal{O}_{X,z} \rightarrow k(x)$ exists due to $z$ being a specialization of $x$. Since all right diagonal arrows are local and hence the composition of the left diagonal arrows with the horizontal arrow is local we get that $\mathcal{O}_{X,z} \rightarrow k(x)$ is local. This implies $\mathcal{O}_{X,z} \rightarrow \mathcal{O}_{X,x}$ is local and by \cite[Corollary ~ 3.11.7 ii)]{Hu1996} this means $z$ is a vertical specialization of $x$.

\end{proof}

\begin{remark}
	
	Assume $f \colon X \rightarrow S$ is locally of weakly finite type and $S$ a stable adic space. Hence $X$ is a stable adic space, which in turn ensures the existence of the fiber product of $f$ with itself. Then
	\begin{enumerate}[label=\roman*)]
		\item $f$ being universally vertically specializing at $x$ is equivalent to having the lifting property with respect to valuation rings $(x,A)$ on $X$ by \cite[Proposition ~ 1.3.8]{Hu1996}. That means if a valuation ring $(x,A)$ of $X$ has a center $y'$ on $S$, then there exists a center $y$ of $(x,A)$ on $X$ with $f(y) = y'$. Note the proof of this proposition works exactly the same without assuming that the giving adic spaces are analytic. 
		
		\item $f$ being separated implies $f$ is vertically separated by \cite[Proposition ~ 1.3.8]{Hu1996}, where the proof for $a) \implies b)$ works exactly the same without any assumption of our adic spaces being analytic.
		
		\item there is an equivalent criterion for quasi-separated and locally of weakly finite type $f$ being separated in \cite[Proposition ~ 3.11.12]{Hu1993} by using valuated valuation rings, where the image under $f$ of two centers of the same valuated valuation ring is equal then the centers were already equal. Note that although in this proposition $f$ is stated as locally of finite type, this is not  necessary and we only need $f$ to be locally of weakly finite type for the proof.
	\end{enumerate}
	
	We will always explicitly mention vertically separated when the full strength of separatedness is not required.
	
\end{remark}

Now we give a modified version of Huber's definition \cite[Definition ~ 5.1.1]{Hu1996} of compactifications for morphisms between adic spaces.

\begin{definition}\label{comp}
	Let $f \colon X \rightarrow S$ be a morphism between adic spaces where $S$ is a stable adic space.
	\begin{enumerate}[label=\roman*)]
		\item A $vertical \ compactification$ of $f$ is a commutative triangle of adic spaces
		\begin{center}
			\begin{tikzcd}
				X \arrow[r, hook, "j"] \arrow[d, "f", swap]  & X' \arrow[dl, "f'"] \\
				S & \\
			\end{tikzcd}
		\end{center}
		where $j$ is a locally closed embedding and $f'$ is vertically partially proper.
		\item A $universal \ vertical \ compactification$ of $f$ is a vertical compactification $(X',f',j)$ of $f$ such that if
		\begin{center}
			\begin{tikzcd}
				X \arrow[dd, "f"] \arrow[dr, "h"] & \\
				& Y  \arrow[dl, "g"] \\
				S & \\
				
			\end{tikzcd}
		\end{center}
		is a commutative triangle of adic spaces with $g$ vertically partially proper, then there exists a unique morphism $i: X' \rightarrow Y$ such that the diagram
		\begin{center}
			\begin{tikzcd}
				X \arrow[dd, "f", swap] \arrow[drr, "h"] \arrow[dr, "j", swap] & & \\
				& X' \arrow[r, "i"] \arrow[dl, "f'", swap] & Y \arrow[dll, "g"] \\
				S & & \\
			\end{tikzcd}
		\end{center}
		commutes.
	\end{enumerate}
\end{definition}

\subsection{Affinoid case}

Already in \cite[Definition ~ 3.13.5]{Hu1993} the notion of the relative closure $\overline{U}^S$ for affinoid analytic spaces $U \rightarrow S$ was defined. We will generalize this notion for arbitrary affinoid adic spaces  and prove some properties for Theorem \ref{main} and Theorem \ref{spa} later.
 In particular Corollary \ref{iso} is needed, because we don't assume that our adic spaces are locally Noetherian. This Corollary will be used in Lemma \ref{mod} i). Theorem \ref{baby} will show for which morphism $U \rightarrow S$ of affinoid adic space a universal vertical compactification exists.

\begin{lemma}\label{gen.hor.}
	We consider a homomorphism of Huber pairs of the form $(A,A_1) \rightarrow (A,A_2)$. Then the image under the induced map $\phi \colon \mathrm{Spa}(A,A_2) \rightarrow  \mathrm{Spa}(A,A_1)$ is closed under generalized horizontal specialization. 
\end{lemma}

\begin{proof}

Let $x_1 \in $ Spa$(A,A_1)$ which is a generalized horizontal specialization of a point $x_2 \in $ Spa$(A,A_2)$. If $x_1$ corresponds to a trivial valuation, then it is clear. If $x_1$ is a horizontal specialization of $x_2$, then $x_1 = x_2 |_H$ for a convex subgroup $H$ of the value group $\Gamma_{x_2}$ with $c\Gamma_{x_2} \subset H$. Hence $x_1(f) \leq 1$ for every $f \in A_2$.

\end{proof}

\begin{proposition}\label{relcl}
	
	Let $f \colon U \rightarrow S$ be a morphism locally of $^+$weakly finite type between affinoid spaces. Denote by $U_c =\mathrm{Spa}(\mathcal{O}_U(U),I(U))$ the closure of $U$ relative to $S$ (sometimes we denote it as $\overline{U}^S$), where $I(U)$ is the integral closure of $\mathcal{O}^+_S(S)[
	\mathcal{O}_U(U)^{\circ \circ}]$ in $\mathcal{O}_X(U)$ and $\phi \colon U \rightarrow U_c$ is the morphism induced by the identity $(\mathcal{O}_U(U), I(U)) \rightarrow (\mathcal{O}_U(U), \mathcal{O}^+_U(U))$. Then the following holds:

	\begin{enumerate}[label=\roman*)]
		\item $\phi \colon U \rightarrow U_c$ is an open immersion,
		\item every point of $U_c$ is a vertical specialization of a point of $\phi(U)$,
		\item $\mathcal{O}_{U_c} \rightarrow \phi_*\mathcal{O}_U$ is an isomorphism of sheaves of topological rings,		
		\item $\phi(U)$ is closed under generalized horizontal specialization.
	\end{enumerate}
	
	\begin{enumerate}[label=\roman*')]
		
		\item if  $\mathcal{O}_U(U)$ is topologically finite type over $\mathcal{O}_S(S)$ and there exists a finite subset $E$ of $\mathcal{O}_U(U)$ such that $\mathcal{O}^+_U(U)$ is the integral closure of $\mathcal{O}^+_S(S)[E \cup \mathcal{O}_U(U)^{\circ \circ}]$, then $\phi$ gives an isomorphism from $U$ onto a rational subset of $U_c$.
	\end{enumerate}
	
\end{proposition}

\begin{proof}

i) It is enough to show that for every rational subset $R(\frac{T}{g})$ in $U$ the image $\phi(R(\frac{T}{g}))$ is open in $U_c$. Let Spa$(A,A^+)$ be an affinoid space which is isomorphic to $U$. We can assume that $A$ is complete. We choose rational subsets $R(\frac{H}{s})$ of $U$ around every point $x$ in $R(\frac{T}{g})$ which are contained in $R(\frac{T}{g})$  such that $R(\frac{H}{s})$ is a witness for $R(\frac{T}{s}) \rightarrow S$ being of $^+$weakly finite type. Let $E \subset A\langle\frac{H}{s}\rangle$ be the finite subset with $A \langle \frac{H}{s} \rangle^+$ is the integral closure of $\mathcal{O}^+_S(S)[E \cup A\langle \frac{H}{s} \rangle^{\circ \circ}]$ . It suffices to show that for the canonical morphism $\psi \colon R(\frac{H}{s}) \hookrightarrow U \rightarrow U_c$ the image is open in $U_c$. We show
\begin{align*}
	\psi(R(\frac{H}{s})) = \{x \in R'(\frac{H}{s}) \ | x(e) \leq 1 \ \text{for every $e\in E$}\}, 
\end{align*}
where $R'(\frac{H}{s})$ is a rational subset in $U_c$. This will show that $\psi(R(\frac{H}{s}))$ is open in $R'(\frac{H}{s})$. Hence, it is open in $U_c$. Because $R(\frac{H}{s}) \cong \mathrm{Spa}(A\langle \frac{H}{s} \rangle, A \langle \frac{H}{s} \rangle^+)$ and the ring of integral elements contains $E$, this explains the equality.

ii) We will distinguish between analytic and non-analytic points: If $x \in U_c$ is an analytic point then by \cite[Remark~7.42 (2)]{Wed19} there is a vertical generalization $y \in U_c$ of rank 1 on $\mathcal{O}_{U_c}(U_c) = \mathcal{O}_U(U)$. Then $y(s) \leq 1$ for every $s \in \mathcal{O}_U(U)^{\circ}$, see \cite[Proposition~7.41]{Wed19}. In particular the inequality holds for every element in $\mathcal{O}_U(U)^+$, hence $y \in \phi(U)$. Now let $x \in U_c$ be a non-analytic point. Then, again by \cite[Remark~7.42 (3)]{Wed19}, there exists a vertical generalization $y$ which is a trivial valuation. Since $x$ is a non-analytic point, the support of $y$ is open. It follows that $y$ is continuous and hence, a point of $\phi(U)$. 
\newline
iii) This is obvious by definition of $U_c$. 
\newline
iv) That is a consequence of Lemma \ref{gen.hor.}.
\bigskip
\newline
$i')$ By \cite[Remark~7.24]{Wed19} we have  $\phi(U)= \{ x \in U_c \ | \ x(f) \leq 1 \ \text{for all } f \in \mathcal{O}^+_U(U) \} = \{ x \in U_c \ | \ x(e) \leq 1 \text \ \text{for all} \ e \in E \}$.
\end{proof}

\begin{proposition}\label{inj}
	
	Let $X \rightarrow S$ be a morphism of adic spaces and $U,V$ open affinoid subspaces of $X$ with $U \subset V$. Then the morphism $U_c \rightarrow V_c$ induced by the restriction map is spectral and injective.
	
\end{proposition}

\begin{proof}

Since $U \hookrightarrow V$ is an open embedding, this map is adic. Hence the restriction map is an adic morphism of Huber rings and this shows $U_c \rightarrow V_c$ is spectral. For the other statement note that by the previous proposition every point in $U_c$ is a vertical specialization of some point in $X$ and $j \colon U \hookrightarrow V \rightarrow V_c$ is just an inclusion and maps vertical generalizations of $x \in U$ surjectively to the vertical generalizations of $j(x)$ by \cite[Proposition ~ 3.8.9]{Hu1993}. Then the injectivity is proven in the proof of \cite[Lemma~ ~ 1.3.14 ii)]{Hu1996}. In this lemma the morphism $g$ will be our morphism $U_c \rightarrow V_c$, the morphism $f$ in this lemma is our $j$, hence $k(x) =k(j(x))$, and we just have to add the word vertical specialization at the right places. We remark that in this proof the lemma \cite[Lemma ~ 1.3.16]{Hu1996} is used for analytic spaces, but in \cite[Proposition ~ 3.11.23]{Hu1996} is a more general version. 

\end{proof}

\begin{lemma}\label{min}
	
	Let $j: X \rightarrow X'$ be an open quasi-compact embedding of adic spaces and suppose that $x' \in X'$ has a vertical generalization in $j(X)$. Then there exists a minimal vertical generalization of $x'$ in $j(X)$. 
	
\end{lemma}

\begin{proof}

Let us denote by $G^\text{v}(x')$ the set of all vertical generalizations of $x'$ in $X'$. We wants to show that $j(X) \cap G^\text{v}(x')$ admits a closed point. Since $j$ is quasi-compact, $j(X)$ is pro-constructible in $X'$. Let $U$ be an open affinoid adic space of $X'$ which contains $x'$. Then $j(X) \cap U$ is pro-constructible in $U$ and $G^\text{v}(x') \subset U$ is isomorphic to the adic spectrum of the valued field of $x'$ and thus, also pro-constructible in $U$. This shows $j(X) \cap G^\text{v}(x')$ is pro-constructible in $U$. It further shows $j(X) \cap G^\text{v}(x')$ is a spectral space, because $U$ is spectral. This concludes our proof.

\end{proof}

\begin{lemma}\label{stalk}
	\begin{enumerate}[label=\roman*)]

		\item 	Let $U \rightarrow S$ be a morphism of $^+$weakly finite type between affinoid adic spaces and $U_c$ defined as like in Proposition \ref{relcl}. Consider an arbitrary element $y \in U_c$ and the minimal element $z$ among of all vertical generalization of $y$ which are contained in $U.$ Then for every open subset $W \subset U$ which contains $z$ there is an open subset $W' \subset U_c$ which contains $y$ such that $W' \cap U = W$.
		
		\item Let $j \colon X \rightarrow X'$ be a quasi-compact open embedding between adic spaces such that $j(X)$ is closed under generalized horizontal specializations in $X'$ and we assume $x' \in X'$ has a vertical generalization in $j(X)$. Let  $j(x) \in j(X)$ be a minimal vertical generalization of $x'$ among all $x \in X$. Let $U$ be an open quasi-compact neighbourhood of $x'$. Then for any open subset $W \subset j(j^{-1}(U))$ which contains $j(x)$ there exists a open subset $W' \subset U$ which contains $x'$ such that $W' \cap j(j^{-1}(U)) = W$.
	\end{enumerate}
\end{lemma}

\begin{proof}
$i)$	Note that the statement in this lemma is equivalent to stating that for every open neighbourhood $W \subset U$ of $z$ the point $y$ is not in the topological closure of $U \backslash W$ in $U_c$. We show the latter statement by contradiction. Assume there is an open affinoid neighbourhood $W \subset U$ of $z$ such that $y$ is an element of the topological closure of $U \backslash W$ in $U_c$. Since $U\backslash W$ is pro-constructible, there is a $w \in U \backslash W$ such that $y$ is a specialization of $w$. Let $h$ be a generalized horizontal specialization of $w$ and a vertical generalization of $y$. Then by Lemma \ref{gen.hor.} we know that $h \in U \backslash W$. Further, $z$ is a specialization of $h$ because $z$ is the smallest vertical generalization of $y$ which is contained in $U$. Hence $z$ is an element of the topological closure of $U \backslash W$ in $U_c$ and we get a contradiction. 
\bigskip

$ii)$ Note that $j(j^{-1}(U))$ is closed under generalized horizontal specializations in $U$. Further in the proof of $i)$ we need that $U$ is a constructible subset in $U_c$. Since $j$ is a quasi-compact open embedding, we have that $j(j^{-1}(U))$ is constructible. Hence we can mimic the proof by replacing $U$ with $j(j^{-1}(U))$ and $U_c$ with $U$.
\end{proof}

\begin{corollary}\label{iso}
	
	\begin{enumerate}[label=\roman*)]
		\item We assume the situation from the Lemma before in i). Then the canonical map $\mathcal{O}_{U_c,y} \rightarrow \mathcal{O}_{U,z}$  is an isomorphism.
		\item Assuming the situation from the lemma before in ii) and that $\mathcal{O}_{X'} \rightarrow j_*\mathcal{O}_{X}$ is an isomorphism of sheaves of topological rings, we get that the canonical map $\mathcal{O}_{X', x'} \rightarrow \mathcal{O}_{X,j(x)}$ is an isomorphism.
	\end{enumerate}
	
\end{corollary}

\begin{proof}

$i)$ We get

\[ \mathcal{O}_{U,z} = \colim_{z \in W \subset U} \mathcal{O}_U(W) = \colim_{z, y \in W' \subset U_c} \mathcal{O}_U(W' \cap U) \cong \colim_{y \in W' \subset U_c} \mathcal{O}_{U_c}(W'),\]

where the second equality holds by Lemma \ref{stalk} i) and the isomorphism at the end holds, because $\mathcal{O}_{U_c}(V) \rightarrow \mathcal{O}_{U}(V \cap U)$ is an isomorphism for every open subset $V$ in $U_c$, see Proposition \ref{relcl} iii). Now note that the family of open neighbourhoods $V \subset U_c$ of $y$ is a cofinal subset of the family of the neighbourhoods $W' \subset U_c$ which contain $z$ and $y$. Hence we get

\[\colim_{z,y \in W' \subset U_c} \mathcal{O}_{U_c}(W') \cong \colim_{y \in V \subset U_c} \mathcal{O}_{U_c}(V) =  \mathcal{O}_{U_c, y}.\]

$ii)$ The proof is similar as before. 

\end{proof}

Now we show the affinoid version of universal vertical compactification. Although we will later show a more general version, we will prove here that $U_c \rightarrow S$ is $^+$weakly of finite type, which will be important for later. For this reason, the second paragraph in the proof can be skipped. In any case, the second paragraph could be proved using the lemmas that we will comment on in detail below in Lemma \ref{mod}. However, for historical reasons, we would like to refer to the original source \cite{Hu1993} where it was mentioned for the first time. Before the proof, we would like to make a remark about the statements we will use.

\begin{remark}
	
	The second paragraph below will use \cite[Lemma~3.13.4(i)]{Hu1993}. Note that in this lemma, this was proved for analytic spaces and for $Y \rightarrow S$ being locally of finite type and separated. But the proof also works for $Y \rightarrow S$ being locally of $^+$weakly finite type and vertically separated. We need locally of $^+$weakly finite type to have an open affinoid subset $V$ of $Y$ such that for a finite subset $E \subset \mathcal{O}_Y(V)$ we get $\mathcal{O}^+_Y(V)$ is the integral closure of $\mathcal{O}^+_S(S)[E \cup \mathcal{O}_Y(V)^{\circ \circ}]$.  Also in said lemma, the proof  \cite[Proposition ~  3.3.13 iii)]{Hu1993} is used for analytic spaces, but it is mentioned in \cite{Hu1993}[Remark ~ 3.6.17] that the necessary part of this proof works also for adic spaces by adding the word "vertical" to "generalization". Also in (b) in the proof of \cite[Lemma ~ 3.13.4 i)]{Hu1993} we have to use that $x \in U$ is the minimal vertical generalization of $z$ to ensure that $\mathcal{O}_{Z,z} \rightarrow \mathcal{O}_{X,x}$ is injective by Corollary \ref{iso} i), see also the discussion below in Lemma \ref{mod} i).  Hence, we can really use this lemma for our theorem.
	
\end{remark}

\begin{theorem}\label{baby}
	Let $f \colon U \rightarrow S$ be a morphism of $^+$weakly finite type between affinoid spaces, where $S$ is a stable adic space. Then $f' \colon U_c \rightarrow S$ is a universal vertical compactification of $f$. 
\end{theorem}

\begin{proof}

First we show that $f'$ is a vertical compactification of $f$. By Proposition \ref{relcl} $i)$ we identify $U$ as an open subset of $U_c$. We want to show that $f'$ is vertically partially proper and in order to do this, we first prove that $f'$ is of $^+$weakly finite type. We may assume that $U$ is attached to a Huber pair $(A,A^+)$ where $A$ is complete. Let $y \in U_c$ be arbitrary and $z$ the minimal element of all vertical generalizations of $y$ which are contained in $U$. We choose an open affinoid subset $V$ of $U$ which is a witness for $U \rightarrow S$ being of $^+$weakly finite type.  By Lemma \ref{stalk} there is an open subset $W'$ in $U_c$ which contains $y$ such that $W' \cap U = V$. Let $Z$ be a rational subset of $U_c$ which is also a subset of $W'$. Then $Z \cap U$ is a rational subset of $U$ and $Z \cap U \subset V$. Therefore, $\mathcal{O}_S(S) \rightarrow  \mathcal{O}_U(Z \cap U) = O_{U_c}(Z) $ is of topologically finite type and since $Z$ is a rational subset of $U_c$, it has the form Spa$(A \langle \frac{T}{g} \rangle, $$\mathcal{O}^+_S(S)[A\langle \frac{T}{g} \rangle^{\circ \circ} \cup \frac{T}{g}]^{\mathrm{int}}$).

We show that $f'$ is vertically separated and universally vertically specializing. For this let $(x,A)$ be a valuation ring on $U_c$ with center $s$ on $S$. We know that morphisms between affinoid spaces are always quasi-separated and since $U_c$ is an affinoid space, by Lemma \ref{criterion} ii) $(x,A)$ has at most one center on $U_c$. Hence $f'$ is vertically separated. Because $(x,A)$ has a center on $S$, we get by Lemma \ref{criterion}i) that the canonical morphism $\mathcal{O}^+_S(S) \rightarrow k(f'(x))$ factors through $A \cap k(f'(x))$. Every valuation ring is open and integrally closed. By construction $\mathcal{O}_{U_c}^+(U_c)$ is the integral closure of $\mathcal{O}^+_S(S)[\mathcal{O}_U(U)^{\circ \circ}]$ in $\mathcal{O}_U(U)$. Hence the canonical morphism $\mathcal{O}_{U_c}^+(U_c) \rightarrow k(x)$ factors through $A$. This shows $f'$ is universally vertically specializing. 
\bigskip

Now we show that $f'$ is indeed universal. Let $Y \rightarrow S$ be a vertically partially proper map such that the first diagram in Definition \ref{comp} ii) commutes  and $Y$ any adic space. We have to show that every valuation ring $(x,A)$ on $U$ which has a center on $U_c$ also has a center on $Y$. Then by  \cite[Lemma ~ 3.13.4 i)]{Hu1993} there exists a unique map $U_c \rightarrow Y$ with the right properties for $U_c$ being a universal vertical compactification. 

Let $(x,A)$ be a valuation ring which has center on $U_c$. Then $(x,A)$ has a center $z$ on $S$. Since $Y \rightarrow S$ is vertically partially proper, hence there exists a center $y$ of $(x,A)$ on $Y$ such that $y$ is mapped to $z$. This concludes the proof. 

\end{proof}

\begin{remark}\label{lift}
	From the proof of this theorem we saw that $f' \colon U_c \rightarrow S$ always has the lifting property with respect to valuation rings $(x,A)$ on $U_c$ which have a center on $S$ and we only need the assumption about $S$ to conclude from this that $f'$ is universally vertically specializing. The assumption about $f$ being of $^+$weakly finite type is only necessary to show that $f'$ is a universal vertical compactification.
\end{remark}

\begin{proposition}\label{homeo}
	Let $X \rightarrow S$ be a morphism of adic spaces with $S$ an affinoid adic space. Let $U,V$ be open affinoid subspaces of $X$ with $U \subset V$. Then the morphism $f \colon U_c \rightarrow V_c$ induced by the restriction map is a topological embedding.
\end{proposition}

\begin{proof}

By Proposition \ref{inj} we already know that $f$ is injective. So it is left to show that the induced map $U_c \rightarrow f(U_c)$ by $f$ is closed. Since $U_c$ is spectral, every closed subset $T \subset U_c$ is pro-constructible. We want to show that $f(T)$ is closed in $f(U_c)$. Since $f(T)$ is pro-constructible, it is enough to show that $U_c \rightarrow f(U_c)$ is specializing. Let us assume $f(s)$ is a specialization of $f(t)$ for $s,t \in U_c$. Since $V_c$ is an adic space, there is a $h \in V_c$ which is a horizontal generalization of $f(s)$ and a vertical specialization of $f(t)$. We show that $h \in f(U_c)$. Let $x \in U$ be a vertical generalization of $t$. Then $f(t)$ is a vertical specialization of $x$, especially $h$ is a vertical specialization of $x$. Note that $f$ is the identity on $U$. Let $(x,A)$ be a valuation ring with center $h$. Then $(x,A)$ has a center on $S$. By the previous remark we can lift the center of $(x,A)$ in $S$ to a center $h'$ of $(x,A)$ on $U_c$. Then $f(h')$ is a center of $(x,A)$ in $V_c$ and since $V_c$ is affinoid, we get $f(h') = h$. 
\bigskip

Now we show that $h'$ is a vertical specialization of $t$. By \cite[Propoisition ~ 3.8.9]{Hu1993} we have $f(G^\text{v}(h)) = G^\text{v}(f(h))$ and thus there is a vertical generalization $t'$ of $h$ with $f(t') = f(t)$. This shows $t = t'$ by injectivity of $f$.
\bigskip

It is left to show that $s$ is a horizontal specialization of $h$. Let $x' \in U$ be a vertical generalization of $s$. Since $V_c$ is affinoid, we get by \cite[Lemma ~ 1.1.15 iv)]{Hu1993} a $x''$ such that we have the following specialization diagram

\begin{center}
	\begin{tikzcd}
		x' \arrow[r, leftsquigarrow] & x'' \\
		f(s) \arrow[u, leftsquigarrow] \arrow[r, leftsquigarrow] & f(h'), \arrow[u, leftsquigarrow]  \\
	\end{tikzcd}
\end{center}
where we mean with $z \rightsquigarrow z'$ that $z'$ is a specialization of $z$ and if we use this arrow within a diagram then its orientation means vertical or horizontal specialization.  Note that $x'' \in U$ because $x' \in U$. Again by \cite[Lemma ~ 1.1.15 i)]{Hu1993}, we get a unique $s'$ in $U_c$ such that 
\begin{center}
	\begin{tikzcd}
		x' \arrow[r, leftsquigarrow] & x'' \\
		s' \arrow[u, leftsquigarrow] \arrow[r, leftsquigarrow] & h'. \arrow[u, leftsquigarrow] \\
	\end{tikzcd}
\end{center}
We get $x''$ is still a vertical generalization of $h'$, because we already showed that $f$ is vertically specializing onto its image.
By applying $f$ we get two diagrams of the form
\begin{center}
	\begin{tikzcd}
		x' \arrow[r, leftsquigarrow] & x'' \\
		f(s') \arrow[u, leftsquigarrow] \arrow[r, leftsquigarrow] & f(h') \arrow[u, leftsquigarrow] \\
	\end{tikzcd}
	\qquad
	\begin{tikzcd}
		x' \arrow[r, leftsquigarrow] & x'' \\
		f(s) \arrow[u, leftsquigarrow] \arrow[r, leftsquigarrow] & f(h'), \arrow[u, leftsquigarrow]  \\
	\end{tikzcd}
\end{center}
which only differ by the lower left corner. By the same lemma as before, the bottom left corner is unique and this implies $f(s')=f(s)$. Thus $s'=s$, because $f$ is injective.  This concludes our proof. 

\end{proof}

\subsection{General case} In this subsection, Theorem \ref{main} shows that a locally of $^+$weakly finite type, separated, and taut morphism $X \to S$, with $X$ a weakly square complete and $S$ a square complete stable adic space, admits a universal vertical compactification. In order to adapt the original proof, we restrict our attention to the class of weakly square complete adic spaces defined in Definition \ref{weaker}. Proposition \ref{weaksqmor} shows that $X'$ of a universal vertical compactification $f' \colon X' \rightarrow S$ of a morphism $f \colon X \rightarrow S$, with $S$ square complete, is necessarily weakly square complete. Furthermore, as in the original proof of Huber $X'$ will be constructed from a gluing datum in Theorem \ref{main}. The gluing datum is given by certain open subsets $U_d$ of $U_c$, see Definition \ref{cmap}, for each open affinoid subset $U$ of $X$. We want to adapt the original proof. As will be shown in Lemma \ref{openset}, $U_d$ is also in the non-analytic case open, if we require that $X$ be weakly square complete. To prove the mentioned lemma, we have to use Lemma \ref{horizontal}.

\begin{remark} We will make some remarks about the used lemmas in the proof of Theorem \ref{main}, which were originally formulated for analytic adic spaces and modify them for our purpose. Especially, we don't need the assumption about being analytic for the adic spaces. Since the modifications are minor and only one point needs to be argued anew at most, we have omitted rewriting the proofs of the lemmas in full in this paper. We assume that the reader is familiar with the lemmas mentioned in the sketch of the proof in the following lemma.

\end{remark}

\begin{remark}
	
	For the proof of the first statement in the Lemma below we usually need the general assumptions about Huber rings which were made by Huber. But we will show that actually those assumptions are not necessary if we make some extra assumptions. Those extra assumptions will naturally hold in Theorem \ref{main}.
	
\end{remark}

\begin{lemma}\label{mod} For the next two points let $S$ be an adic space, $X'$ and $Y$ adic $S$-spaces, $j \colon X \hookrightarrow X'$ a quasi-compact open immersion such that $X$ is closed under generalized horizontal specializations in $X'$ and $f \colon X \rightarrow Y$ a $S$-morphism.
	\begin{enumerate}[label=\roman*)] 
		
		\item If every point of $X'$ is a vertical specialization of a point of $X$, $\mathcal{O}_{X'} \rightarrow j_*\mathcal{O}_X$ is an isomorphism of sheaves of topological rings and $Y$ is vertically separated over $S$, then there exists at most one $S$-morphism $g \colon X' \rightarrow Y$ with $f = g |_X.$

		\item We assume that $Y$ is locally of $^+$weakly finite type over $S$ and that for every open subset $V$ of $X'$ the restriction mapping $\mathcal{O}_{X'}(V) \rightarrow \mathcal{O}_X(V \cap X)$ is an isomorphism of topological rings. Let $x'$ be a point of $X'$ and $y$ a point of $Y$ such that there exists a valuation ring $(x,A)$ of $X$ such that $x'$ is a center of $(x,A)$ on $X'$, $y$ is a center of $(x,A)$ on $Y$, $x'$ and $y$ lie over the same point of $S$ and $x$ is minimal (closed point) among all vertical generalizations of $x'$ in $X$. Then there exist an open neighbourhood $U$ of $x'$ in $X'$ and an $S$-morphism $g \colon U \rightarrow Y$ with $g|_{U \cap X} = f|_{U \cap X}$ and $y=g(x')$.
		
		\item Let $f \colon X \rightarrow Y$ be a morphism of adic spaces with $X$ a stable adic space and let $u,v \in X$ be such that $u$ is a vertical generalization of $v$. Then $f$ is universally vertically specializing at $v$ if $f$ is vertically separated and universally vertically specializing at $u$.
		
		\item Let 
		\begin{center}
			\begin{tikzcd}
				X \arrow[d, "f"] \arrow[r, hookrightarrow, "j"] & X' \arrow[dl, "f'"] \\
				S & \\
			\end{tikzcd}
		\end{center}
		be a commutative diagram of adic spaces such that $f'$ is vertically partially proper and $j$ is a quasi-compact open embedding such that $j(X)$ is closed under generalized horizontal specializations, every point of $X'$ is a vertical specialization of a point of $j(X)$ and $\mathcal{O}_{X'} \rightarrow j_*\mathcal{O}_X$ is an isomorphism of sheaves of topological rings. Then for every open subset $U$ of $S$ the diagram induced from the diagram above
		\begin{center}
			\begin{tikzcd}
				f^{-1}(U) \arrow[d] \arrow[r, hookrightarrow] & f'^{-1}(U)  \arrow[dl]\\
				U & \\
				
			\end{tikzcd}
		\end{center}
		
		is a universal vertical compactification of $f \colon f^{-1}(U) \rightarrow U$.
		
	\end{enumerate}
\end{lemma}

\begin{proof}
	Let $d \colon X' \rightarrow S$ and $e \colon Y \rightarrow S$ be the structure morphisms.
	\\
	
	$i)$  We explain why the proof in \cite[Lemma ~ 1.3.14 i)]{Hu1996} still works: We assume $x \in X$ is the minimal vertical generalization of $x' \in X'$, which exists by Lemma \ref{min}. Now the proof works exactly the same by adding the word vertical to the word generalization. We wants to show that $g$ is unique. As map between the topological spaces of $X'$ and $Y$ it is unique because $e$ is vertically separated, see for this the original proof. Now we have to show that the induced morphism $g^* \colon \mathcal{O}_{Y,g(x')} \rightarrow \mathcal{O}_{X', x'}$ on the stalks is unique. For this we consider the same diagram as in the original proof: 
	\begin{center}
		\begin{tikzcd}
			\mathcal{O}_{X,x} & \mathcal{O}_{X',x'} \arrow[l, "\alpha", swap] \\
			\mathcal{O}_{Y,f(x)} \arrow[u, "f^*"] & \mathcal{O}_{Y, g(x')} \arrow[l] \arrow[u, "g^*", swap] \\
		\end{tikzcd}
	\end{center}

	 Note that the morphism $\alpha \colon \mathcal{O}_{X',x'} \rightarrow \mathcal{O}_{X,x}$ in this lemma will be even an isomorphism by Corollary \ref{iso}. The reason for the additional assumptions is to use this corollary.
	We needed all the further assumptions to ensure $\alpha$ to be at least injective. Since we do not use the general assumptions of Huber about Huber rings, in general the morphism $\alpha$ is not flat. Hence, we can not use the usual argument that $\alpha$ is local and flat and thus injective.
	\\
	
	ii) Based on the proof of \cite[Lemma ~ 1.3.14 iii)]{Hu1996}: We copy the beginning of the original proof. We choose affinoid open subsets $L$ and $M$ of $Y$ and $S$ such that $y \in L$, $e(L) \subset M$ and there exists a finite subset $E$ of $\mathcal{O}_{Y}(L)$ such that $\mathcal{O}^+_Y(L)$ is the integral closure of $\mathcal{O}^+_S(M)[E \cup \mathcal{O}_Y^{\circ \circ}]$. So far we didn't add anything new to the proof. Now we have to modify the statement (1) in this proof to:
	\bigskip
	\newline
	Let $G^v$ denote the set of the vertical generalizations of $x'$ in $X'$. Then $G \cap X \subset f^{-1}(L).$
	\bigskip
	\newline
	This is proven the same as in the original proof where we reading specialization as vertical specialization. But now we have to argue a bit different why this statement implies
	 the existence of an affinoid open neighbourhood $V$ of $x'$ in $d^{-1}(M)$ with $V \cap X \subset f^{-1}(L)$: Choose an open quasi-compact neighbourhood $W$ of $x'$ in $d^{-1}(M)$. We consider $W \cap X$, which is quasi-compact. Now we use Lemma \ref{stalk} ii) and get for $f^{-1}(L) \cap W \cap X$ an open neighbourhood $V'$ of $x'$ with $f^{-1}(L) \cap W \cap X = V'$. By choosing an affinoid open neighbourhood $V$ of $x'$ in $f^{-1}(L) \cap W$ we get the desired statement. The rest of the proof remains unchanged.
	\\
	
	iii) In \cite[Lemma ~ 1.10.23 ii)]{Hu1996} we add again the word vertical whenever it makes sense, the proof goes exactly as stated there. Note that the category of adic spaces is a full subcategory of pseudo-adic spaces by \cite[Remark ~ 1.10.2]{Hu1996}.
	\\
	
	iv) That is almost like \cite[Lemma ~ 5.1.7 ]{Hu1996}, we just add again vertical at the right places in the proof. The further assumption about  $j(X)$ being closed under generalized horizontal specializations is needed to use the modified versions of \cite[Lemma ~ 1.3.14 i)]{Hu1996} and \cite[Lemma ~ 1.3.14 iii)]{Hu1996} which we mentioned above.  There is nothing more to add to the proof itself. 
	
\end{proof}
We define a special case of \cite[Definition ~ 3.13]{Hüb25}:

\begin{definition}\label{weaker}
	\begin{enumerate}[label=\roman*)]

		\item An adic space $X$ is called $square$ $complete$ if for any three points $x,x_1$ and $x_2$ in $X$ such that $x_1$ is a horizontal specialization and $x_2$ is a vertical specialization of $x$ there exists a unique point $x'$ in $X$ such that $x'$ is a vertical specialization of $x_1$ and a horizontal specialization of $x_2$. 
		
		\item If the point $x'$ is not necessarily unique then $X$ is called $weakly$ $square$ $complete$.
	\end{enumerate}
\end{definition}

\begin{proposition} \label{weaksqmor}
	Let $X$ be a weakly square complete, $S$ be a square complete adic space and $f \colon X \rightarrow S$ a morphism of adic spaces. Furthermore, let $S'$ be an open weakly square complete subspace of $S$. Then $f^{-1}(S')$ is weakly square complete subspace of $X$.
\end{proposition}

\begin{proof}
	 Let $x_1,x$ and $x_2$ be points in $f^{-1}(S')$ such that $x$ is horizontal generalization of $x_1$ and a vertical generalization of $x_2$. Since $X$ is weakly square complete there exists a point $x' \in f^{-1}(S')$ which is a vertical specialization of $x_1$ and horizontal specialization of $x_2$. Since $f$ preserves horizontal and vertical specializations, we get a square of specializations where the left bottom corner ist $f(x')$. Because $S'$ is weakly square complete and $S$ square complete, we get by uniqueness that $f(x')$ is in $S'$. This shows that the point $x'$ is in $f^{-1}(S')$ and this implies that $f^{-1}(S')$ is weakly square complete.
\end{proof}

\begin{remark}\label{newdef}
	The previous proposition motivates the following definition: A morphism $f \colon X \rightarrow S$ between adic spaces is called $weakly$ $square$ $complete$ if for every open weakly square complete subspace $S'$ of $S$ the preimage $f^{-1}(S')$ is weakly square complete in $X$. Then in Theorem \ref{main} we could drop the condition on $X$ and $S$, replacing them by $f$ being weakly square complete.
\end{remark}

\begin{proposition} \label{necessary}
	
	Let $f \colon X \rightarrow S$ be morphism of adic spaces where $S$ is a square complete and stable adic space such that $f$ admits a universal vertical compactification $f' \colon X' \rightarrow S.$ Then $X'$ is weakly square complete. 
	
\end{proposition}

\begin{proof}

Let $x_1, x$ and $x_2$ be points of $X'$ such that $x_1$ is a horizontal specialization and $x_2$ a vertical specialization of $x$. We consider $f'(x_1), f'(x)$ and $f'(x_2)$ and a point $d$ in $S$ which is a vertical specialization of $f'(x_1)$ and horizontal specialization of $f'(x_2)$. Such a point $d$ exists because $S$ is square complete. We choose an open affinoid neighbourhood $S_i$ of $d$. In $f'^{-1}(S_i)$ we choose open affinoid neighbourhood $U$ of $x_1$ and $V$ of $x_2$. Since the restriction map $f'^{-1}(S_i) \rightarrow S_i$ is vertically partially proper and $S_i$ is affinoid we get $U \cap V$ is affinoid. Furthermore, we get a unique map $\overline{U \cap V}^{S_i} \rightarrow f'^{-1}(S_i)$ and $\overline{U}^{S_i} \rightarrow f'^{-1}(S_i)$ such that the following diagram
\begin{center}
	\begin{tikzcd}
		
		& &  U \cap V \arrow[d, hookrightarrow]  \arrow[ddll, hookrightarrow] &  \\ 
		& & U \arrow[dl, hookrightarrow] \arrow[d] \arrow[dr, hookrightarrow] &  \\
		\overline{U \cap V}^{S_i} \arrow[r] \arrow[drr] & \overline{U}^{S_i} \arrow[dr] \arrow[r] & f'^{-1}(S_i) \arrow[d] \arrow[r, hookrightarrow] & X' \arrow[d] \\
		& & S_i \arrow[r, hookrightarrow] & S  \\
		
	\end{tikzcd}
\end{center}

is commutative. Let $(x,A)$ be a valuation ring which has $x_2$ as center. In particular, this valuation ring has $f'(x_2)$ as center on $S_i$ and we can lift it to a center $z$ of $(x,A)$ on $\overline{U \cap V}^{S_i}$. Then we consider the image $z'$ of $z$ under the map $\overline{U \cap V}^{S_i} \rightarrow \overline{U}^{S_i}$. In $\overline{U}^{S_i}$ we have $x_1$ is a horizontal specialization and $z'$ is a vertical specialization of $x$. Since affinoid spaces are square complete, we obtain a point $y$ in $\overline{U}^{S_i}$ such that $y$ is a vertical specialization of $x_1$ and horizontal specialization of $z'$. The commutativity of the diagram above ensures that the image of $y$ under the map $\overline{U}^{S_i} \rightarrow f'^{-1}(S_i)$ is a vertical specialization of $x_1$ and a horizontal specialization of $x_2$.

\end{proof}

The following definitions were part of Huber's original proof of the existence of a universal compactification. We present them here as a separate definition for improved readability.

\begin{definition} \label{cmap} Let $X \rightarrow S$ be a locally of $^+$weakly finite type and separated morphism of adic spaces with $S$ an affinoid adic space.
	We denote by $\mathcal{F}$ the set of all open affinoid subsets of $X$. For each $U \in \mathcal{F}$ we define 
	\begin{enumerate}[label=\roman*)]
		\item 

		\begin{enumerate}[label=(\arabic*)]
		\item $R(U) \coloneqq$ \{$A \in U_{\nu}$ \ | \ $A$ has a center on $S$\} = \{$A \in U_{\nu}$ \ | \ $A$ has a center on $U_c$ \} $\subset U_{\nu}$, 
		
		\item $S(U)\coloneqq \{A \in U_{\nu}$ \ | \ $A$ has a center on $X$\} $\subset R(U)$,
		
		\item $T(U) \subset R(U)$ denotes the set of all $(x,A) \in R(U)$ with the property, that for every valuation ring $B$ of $k(x)$ with $A \subset B \subset k(x)^+$, either the valuation ring $(x,B)$ of $U$ has no center on $X$ or $(x,B)$ has a center on $U$,
		\end{enumerate}
		where we consider $A \in U_{\nu}$ a valuation ring with support in $U$.
	
	\item the map $c_U \colon R(U) \rightarrow U_c$ which assigns to every $A \in R(U)$ the center of $A$ of $U_c$,
	
	\item $U_d \coloneqq c_U(T(U)).$

	\end{enumerate}

\end{definition}

The lemma below was also part of Huber's original proof, but we have modified part (ii). We avoided proving that 
$S(U)$ is pro-constructible. All other parts of the lemma are proven in the same way. We only carefully argued why all statements remain valid when considering exclusively vertical specializations.

\begin{lemma}\label{cmapprop}
	Let $X \rightarrow S$ be a locally of $^+$weakly finite type, separated and taut morphism between adic spaces where $S$ is an affinoid stable adic space. 
		We show that $c_U$ is well-defined, the equality of (1) in Definition \ref{cmap} and
	for every $U, V \in \mathcal{F}$ we have
	\begin{enumerate}[label=\roman*)]
		\item $c_U \colon R(U) \rightarrow U_c$ is surjective,
		\item $c_U(S(U))$ is a pro-constructible subset of the spectral topological space $U_{c}$,
		\item if $U \subset V$ then $T(U) \subset T(V)$,
		\item $U \cap V \in \mathcal{F}$ and $T(U \cap V) = T(U) \cap T(V)$,
		\item $T(U) = c^{-1}_U(c_U(T(U)))$,
		\item if $U \subset V$ then $T(U) = c^{-1}_V(c_V(T(U)))$,
		\item if $U \subset V$ then $R(U) \cap T(V) = c^{-1}_U(c_U(R(U) \cap T(V)))$.
	\end{enumerate}
\end{lemma}

\begin{proof}
	
	Note that the equality in (1) in Definition \ref{cmap} is a consequence of $U_c \rightarrow S$ being a universally vertically specializing map which was already shown in the proof of Theorem \ref{baby} for the affinoid universal vertical compactification.
	The map $c_U$ is well-defined since a valuation ring on an affinoid adic space has at most one center.
	
	i) If $x \in U_c$ then by Proposition \ref{relcl} ii) there is a $y \in U$ which is a vertical generalization of $x$, thus due to Lemma \ref{criterion} there exists a valuation ring of $U$ which has $x$ as center.
	\bigskip
	\newline
	ii) Since $S$ is affinoid and $f$ is taut, we get that $X$ is also taut. Hence there exists a finite open affinoid cover $U_1,...,U_n$ of $\overline{U}$ with $U_1,...,U_n \in \mathcal{F}$. Note that by $(1, iv)$ we have $U_i \cap U \in \mathcal{F}$. Let $\alpha \colon (U_i \cap U)_c \rightarrow (U_i)_c$ and $\beta \colon (U_i \cap U)_c \rightarrow U_c$ be the morphisms which are induced by the restriction maps. By Proposition \ref{inj} we have $\alpha$ and $\beta$ are injective and spectral morphism. \\ Consider $S_i \coloneqq \{(x,A) \in R(U) \ | \ (x,A) \ \text{has a center on $U_i$} \}$. It is easy to see that $c_{U_i \cap U} (S_i) = \alpha^{-1}(c_{U_i}(S_i))$ and $c_U(S_i) = \beta(\alpha^{-1}(c_{U_i}(S_i)))$. We want to show $c_{U_i}(S_i)$ is pro-constructible, as then $c_U(S_i)$ is also pro-constructible by the equation before because $\beta$ and $\alpha$ are spectral mappings. This shows $c_U(S(U))$ is pro-constructible since it is a finite union of pro-constructible sets. It is clear that $c_{U_i}(S_i)$ is pro-constructible if $c_{U_i}(S_i) = U_i \cap \alpha((U_i \cap U)_c)$ holds: For the left to right inclusion consider some $z \in (U_i)_c$ which is a center of some $(x,A) \in S_i$, so $z \in U_i$. Further we have that $\alpha ((U_i \cap U)_c) \supset \alpha(c_{U_i \cap U} (S_i)) = c_{U_i}(S_i)$. Altogether this shows $c_{U_i}(S_i) \subset U_i \cap \alpha((U_i \cap U)_c)$. If $z$ is from the right set then it is the center of a valuation ring with support in $U_i \cap U$ which has a center on $U_i$, this shows the other inclusion.
	\bigskip
	\newline
	iii) Let $(x,A)$ be in $T(U)$ and suppose $(x,A) \notin T(V)$. Then there exists a valuation ring $A \subset B \subset k(x)^+$ such that $(x,B)$ has a center on $X$ and not on $V$. But this implies that $(x,B)$ has also no center on $U$ and we reach a contradiction.
	\bigskip
	\newline
	iv) The first statement follows because the intersection of two open affinoid subspaces is again open affinoid if $S$ is affinoid and $f$ is separated. Due to iii) it is clear that $T(U \cap V) \subset T(U) \cap T(V)$. For the other inclusion let $(x,A)$ be in $T(U) \cap T(V)$ and $A \subset B \subset k(x)^+$. Suppose $(x,B)$ has a center on $X$. Then it has a center on $U$ and a center on $V$. As $S$ is an affinoid adic space and $X \rightarrow S$ is vertically separated, the two centers must coincide.
	\bigskip
	\newline
	v) It follows from vi). 
	\bigskip
	\newline
	vi) Let $(x,A)$ be an element of $c^{-1}_V(c_V(T(U)))$ and suppose that $(x,A) \notin T(U)$. We have $c_V((x,A)) =c_V((x',A'))$ for some $(x',A') \in T(U)$. So those two valuation rings have a common center on $V_c$ and in particular $x$ and $x'$ are vertical generalizations of a common point. By  \cite[Lemma~3.6.16 ii)]{Hu1993} vertical generalizations of a point form a chain. Hence we are in one of the following cases:
	\bigskip
	\newline
	a) $x$ is a vertical specialization of $x'$: Consider the canonical morphism $i \colon k(x) \rightarrow k(x')$. Then Lemma \ref{center} iii) implies $i^{-1}(A') = A$. Let $(x,B)$ with $A \subset B \subset k(x)^+$ be a valuation ring which has a center on $X$ and not on $U$. By \cite[~1.1.14 e)]{Hu1996} there exists a valuation ring $B'$, which contains $A'$, such that $i^{-1}(B') = B$. By Lemma \ref{val} ii) we know that $B' \subset k(x')^+ $ and thus $(x',B')$ is a valuation ring of $X$. We get a contradiction, because $(x',B')$ has the same center as $(x,B)$, but the center is not on $U$ and $(x',A') \in T(U)$. In particular this shows $x \in U$.
	\newline
	b) $x'$ is a vertical specialization of $x$: Similarly to before we consider the canonical morphism $i: k(x') \rightarrow k(x)$ and we get $i^{-1}(A) = A'$. Let $C$ be a valuation ring with support $x$ such that $x'$ is a center of $(x,C)$. Since $i^{-1}(C)= k(x')^+$ and $i^{-1}(A) \subset k(x')^+$ we get $A \subset C$ again by Lemma \ref{val} ii). Suppose $A \subset B \subset k(x)^+$ is a valuation ring which has a center $w$ on $X$ and not on $U$. We have $B \subset C$ or $C \subset B$, because valuation rings which contain a common valuation ring are totally ordered. If $C \subset B$, then Lemma \ref{criterion} i) yields that $B$ has a center on $U$, which cannot be. Now let $B \subset C$. By Lemma \ref{val} i)  $w$ is a vertical specialization of $x'$. Then there exists a valuation ring $(x',C')$ which has center $w$. Since $A' \cap k(w) = (A \cap k(x)) \cap k(w) = A \cap k(w)  \subset B \cap k(w) = k(w)^+ = C' \cap k(w),$ we get $A' \subset C'$, which is a contradiction since $(x',A') \in T(U)$ and $(x',C')$ has a center on $X$ and not on $U$. 
	\bigskip
	\newline
	vii) Let $(x,A) \in c^{-1}_U(c_U(R(U) \cap T(V)))$, especially $(x,A) \in R(U)$ and $(x,A) \in c^{-1}_V(c_V(T(V)))$. By Lemma \ref{cmapprop} v) we have $(x,A) \in T(V)$ and this shows vii). 
	
\end{proof}

\begin{lemma}\label{horizontal}
	Let $X \rightarrow S$ be a morphism between adic spaces, where $S$ is an affinoid space and $X$ is a weakly square complete adic space. Further let $U$ be an open affinoid subset of $X$ and $z, h \in U_c$ such that $z$ is a horizontal specialization of $h$ and there exists a valuation ring $(x_z,A)$ with support in $U$ and center $z$ and we assume that for every valuation ring $ A \subset B \subset k(x_z)^+$ the valuation ring $B$ has no center on $X$ or it has a center on $U$. Then there exists a valuation ring $(x_h,A')$ on $U$ with center $h$ which has the same property as $(x_z,A)$.
\end{lemma}

\begin{proof}
	
	Since $U_c$ is an affinoid space, by \cite[Lemma ~ 1.1.15 iv)]{Hu1993} there exists a $x_h$ such that $x_h$ is a vertical generalization of $h$ and a horizontal generalization of $x_z$. Note $x_h \in U$ since $x_z$ is it too. Let $A'$ be a valuation ring with support $x_h$ and center $h$. We show $(x_h,A')$ is the valuation ring which we are looking for. Let us assume there exists a valuation ring $A' \subset B \subset k(x_h)^+$ which has a center on $X$ but not on $U$. We consider the following diagram:
	
	\begin{center}
		\begin{tikzcd}
			k(x_z)^+ \arrow[r, "\pi_z"] & k(x_z)^{\succ} \arrow[r, "\phi"] & k(x_h)^{\succ} & k(x_h)^+ \arrow[l, "\pi_h"] \\
			k(z)^+ \arrow[u, "\iota_z"] \arrow[r, "\tau_z"] & k(z)^{\succ} \arrow[r, "\psi"] & k(h)^{\succ} & k(h)^+ \arrow[l, "\tau_h"] \arrow[u, "\iota_h"] 
		\end{tikzcd}
	\end{center}
	
	The proof is done, if we show 
	\begin{align*}
		(\phi \circ \pi_z )(A) \subset \pi_h (A').
	\end{align*}
	Firstly, let us discuss why it is enough to show this inclusion. From this inclusion it follows that $A \subset \pi_z^{-1}(\phi^{-1}(\pi_h(A'))) \subset \pi_z^{-1}(\phi^{-1}(\pi_h(B)))  \subset k(x_z)^+$. We show $(x_z,  \pi_z^{-1}(\phi^{-1}(\pi_h(B)))$ has a center on $X$. We know that $(x_h,B)$ has a center on $X$. Since $X$ is square complete, there is an open affinoid neighbourhood $V \coloneqq$ Spa$(R,R^+)$ which contains $x_z, x_h$ and the center of $(x_h,B)$.  The following diagram 
	
	\begin{tikzcd}
		&	k(x_z)^{\succ} \arrow[rr, "\phi"] & & k(x_h)^{\succ} & \\
		\pi_z^{-1}(\phi^{-1}(\pi_h(B))) \arrow[r, hook] \arrow[ur] &	k(x_z)^+ \arrow[u, "\pi_z"] & & k(x_h)^+ \arrow[u, "\pi_h"] & B \arrow[l, hook'] \arrow[ul]  \\
		&	\mathcal{O}^+_{V,x_z} \arrow[u] \arrow[rr] & & \mathcal{O}^+_{V,x_h} \arrow[u] & \\
		&	& \mathcal{O}^+_V(V) \arrow[ul] \arrow[uull, dashed, bend left] \arrow[uurr, dashed, bend right] \arrow[ur]  & & \\
	\end{tikzcd} 
	
	commutes. Note that we get the right dashed arrow because $(x_h,B)$ has a center on $V$. We get the left dashed arrow, because the middle diagram commutes and because of the right dashed arrow. We show $(x_z, \pi_z^{-1}(\phi^{-1}(\pi_h(B))))$ has no center on $U$. Let us assume otherwise. We have to consider the same diagram as above, but we replace the stalks on $V$ by the stalks on $U$ and $\mathcal{O}^+_V(V)$ by $\mathcal{O}^+_U(U)$. Now we get first the left dashed arrow by assumption. Similarly since the middle diagram commutes and because of the left dashed arrow, we get the right dashed arrow. Note that $B = \pi_z^{-1}(\pi_z(B))$ by Remark \ref{spec. field} iii). This implies $(x_h,B)$ has a center on $U$ which is a contradiction. So it is indeed enough to show the inclusion above. To show this inclusion it is enough to show $(\phi \circ\pi_z (\iota_z((k(z)^+))) \subset \pi_h(\iota_h(k(h)^+))$. Then $\iota_z(k(z)^+)  \subset \pi_z^{-1}(\phi^{-1}(\pi_h(\iota_h(k(h)^+)))) \subset \pi_z^{-1}(\phi^{-1})(\pi_h(A')).$ The last set is clopen in $k(x_z)^+$. Because $z$ is a center of $A$, we have $\iota_z(k(z)^+) = A \cap \iota_z(k(z))$. Hence by a similar argument as in Lemma \ref{val} ii) we have $A \subset \pi_z^{-1}(\phi^{-1}(\pi_h(A')))$, which shows the desired inclusion. So it is left to show $(\phi \circ\pi_z (\iota_z((k(z)^+))) \subset \pi_h(\iota_h(k(h)^+))$.  This is obvious by observing the following commutative diagram
	
	\begin{center}
		\begin{tikzcd}
			k(x_z)^{\succ} \arrow[rr, "\phi"] & &  k(x_h)^{\succ} & &\\
			&	k(z)^+ \arrow[ul, "\pi_z \circ \iota_z"] & &   k(x_h)^+ \arrow[ul, "\pi_h"] & \\ 
			& & & \mathcal{O}^+_{U_c,x_h} \arrow[u, twoheadrightarrow] & \\
			&  \mathcal{O}^+_{U_c,z} \arrow[uu, twoheadrightarrow] \arrow[rr]  & & \mathcal{O}^+_{U_c,h} \arrow[r, twoheadrightarrow] \arrow[u] & k(h)^+.  \arrow[uul, bend right, "\iota_h"] \\
		\end{tikzcd}
	\end{center}
	
\end{proof}

The lemma below is restated as part of Huber's original proof, but here we provide additional proofs in (ii), (iii), and (v) to cover the case of horizontal specializations. In particular, the proof of (ii) relies on the lemma stated above.

\begin{lemma}\label{openset} Let $X \rightarrow S$ be a locally of $^+$weakly finite type, separated and taut morphism of adic spaces with $X$ weakly square complete and $S$ an affinoid stable adic space.
	 For $U,V \in \mathcal{F}$ with $U \subset V$, let $\psi_{V,U}: U_c \rightarrow V_c$ be the morphism of adic spaces which is induced by the restriction mapping $(\mathcal{O}_X(V),I(V)) \rightarrow (\mathcal{O}_X(U),I(U)).$ Then
	\bigskip

	 For $U,V \in \mathcal{F}$ with $U \subset V$ we have
	\begin{enumerate}[label=\roman*)]
		\item $U \subset U_d$, 
		\item $U_d$ is open in $U_c$,
		\item $\psi_{V,U}(U_d) \subset V_d$ and $\psi_{V,U} \colon U_d \rightarrow V_d$ is an open embedding of adic spaces,
		\item $\psi_{V,U}(U_d \backslash U)) \subset V_d \backslash V$,
		\item $\psi_{V,U}(U_d \backslash U )$ is closed under specialization in $V_c$. 
	\end{enumerate}
	
\end{lemma}

\begin{proof}
	i) For every $x \in U$, the center of $(x, k(x)^+)$ is $x$ and $(x,k(x)^+) \in T(U)$.
	
	ii) Let $s(c_U(S(U)) \backslash U)$ denote the set of the specializations of the points of $c_U(S(U)) \backslash U$ in $U_c$.
	
	Then
		\begin{align}	
			U_d = U_c \backslash s(c_U(S(U)) \backslash U).	\tag{*} \label{U_d}
		\end{align}

	Let $z$ be an element of $ U_c \backslash s(c_U(S(U)) \backslash U)$. By Lemma \ref{cmapprop} i) the map $c_U$ is surjective and we can assume that $z =c_U((x,A))$ for some $(x,A) \in R(U)$. Suppose $(x,A) \notin T(U)$. Then there exists a valuation ring $A \subset B \subset k(x)^+$ which has a center on $X$ and not on $U$. Thus $ z' \coloneqq c_U((x,B)) \in c_U(S(U)) \backslash U$. By Lemma \ref{val} i) we get that $z$ is a specialization of $z'$ what gives the desired contradiction. For the other inclusion let $z$ be in $U_d$ and suppose $z \in s(c_U(S(U))\backslash U)$. Then $z$ is a specialization of a point $y \in c_U(S(U)) \backslash U$. Let $h \in U_c$ be a horizontal generalization of $z$ and a vertical specialization of $y$. By Lemma \ref{horizontal} we have $h \in U_d$. We consider again $y$ which is a center of a valuation ring $(x,A)$ which has a center on $X$ and not on $U$. By Lemma \ref{center} ii) we get a valuation ring $(x,B)$ which has $h$ as center and it follows immediately by the proof of this lemma that $B \subset A$. Lemma \ref{cmapprop} v) implies $(x,B) \in T(U)$ and that is a contradiction, because $A$ is a valuation ring which contains $B$ and has a center on $X$ and not on $U$.  
	By Lemma \ref{cmapprop} ii) we have $c_U(S(U))$ pro-constructible. Since $U$ is a constructible subset of $U_c$ by Proposition \ref{relcl} i), we obtain $c_U(S(U)) \backslash U$ is a pro-constructible subset of $U_c$. Then by \cite[Proposition~3.30 (2)]{Wed19} the subspace $s(c_U(S(U))\backslash U)$ is closed in $U_c$ and \eqref{U_d} implies $U_d$ is open in $U_c$.
	
	iii) By Lemma \ref{cmapprop} iii) the inclusion $\psi_{V,U}(U_d) \subset V_d$ holds.  Similarly as for \eqref{U_d} let $s(c_V(S(V)) \backslash U)$ $\subset V_c$ denote the set of the specializations of the points of $c_V(S(V) \backslash U)$ in $V_c$. We show 
	\begin{align}
	 \psi_{V,U}(U_d) = V_c \backslash s(c_V(S(V)) \backslash U). \tag{*'} \label{imU_d}
	\end{align}
	
	If $z$ is from the right set, then $z$ must be a vertical specialization of a $x \in U$. Let $(x,A)$ be a valuation ring with center $z$. By the same argument as in $(3)$ we get $(x,A) \in T(U)$. Further by the same argument as in Proposition \ref{homeo} we get $\psi_{V,U}(z') = z$ for a $z' \in U_d.$ Now we consider $z$ as an element from the set on the left in \eqref{imU_d}. If $h$ is a horizontal generalization of $z$, we can show in a similar way as like in Lemma \ref{horizontal} that every valuation ring $(x,A)$ with support in $U$ and center $h$ is an element of $T(U)$. Hence we can assume $z$ is a vertical specialization of a element of $s(c_V(S(V) \backslash U))$.  This will lead to a contradiction like in \eqref{U_d} by using Lemma \ref{cmapprop} vi) this time. 
	\newline
	
	This shows that $\psi_{V,U}(U_d)$ is open in $V_d$. But this also shows $\psi_{V,U}: U_d \rightarrow V_d$ is an open embedding of topological spaces because by Proposition \ref{homeo} the map $\psi_{V,U} \colon U_c \rightarrow V_c$ is a homeomorphism onto its image. Now we consider $U_d$ as open subspace of $V_d$. Let $W$ be an open subset of $\psi_{V,U}(U_d)$. We have $\psi^{-1}_{V,U}(W) \cap U = W \cap U = W \cap V$, where the first equality holds due to $\psi_{V,U}$ being the identity on $U$ and the second equality follows by (iv).  It follows that $\mathcal{O}_{U_c}(\psi^{-1}_{V,U}(W)) \cong \mathcal{O}_U(\psi^{-1}_{V,U}(W) \cap U) \cong \mathcal{O}_V (W \cap V) \cong \mathcal{O}_{V_c}(W)$, where the first and last isomorphism follows by Proposition \ref{relcl}. This shows that $\psi_{V,U}:U_d \rightarrow V_d$ is an open embedding of adic spaces.
	
	iv) This follows by definition of $U_d$ and $V_d$ and since $\psi_{V,U}$ on $U$ is just the identity. 
	
	v) It is enough to show that $\psi_{V,U}(U_d \backslash U)$ is closed under vertical and horizontal specializations in $V_c$. Note that we will consider $U_d$ as open subspace of $V_d$ to save some notation. Otherwise to show that the point will really be in the image is always a similar argument as in Proposition \ref{homeo} which we used before. First we assume that $z \in V_c$ is a vertical specialization of some $y \in U_d \backslash U$. Since $y \in U_d$, there exists some $(x,A) \in T(U)$ such that $c_U((x,A)) = y$ and in particular $y$ is a vertical specialization of $x$. This yields that $z$ is a vertical specialization of $x$. Using Lemma \ref{center} iii) we find a valuation ring $(x,B)$ with center $z$. Observe $$k(z) \cap B = k(z)^+ \subset k(y)^+ = k(y) \cap A \subset A.$$ Lemma \ref{val} implies $B \subset A$. It suffices to show that $(x,B) \in T(U)$. Consider a valuation ring $C$ with $B \subset C \subset k(x)^+$ with $(x,C)$ has a center on $X$ but not on $U$. Since $C$ and $A$ contain $B$, we get $C \subset A$ or $A \subset C$. The scenario $C \subset A$ cannot hold, because if $(x,C)$ has a center on $X$ then $(x,A)$ has also a center on $X$. This center is not in $U$, otherwise $y \in U$. The case $A \subset C$ can also not hold because $(x,A) \in T(U)$. This shows that such a $C$ cannot exist, $(x,B) \in T(U)$ and $z \in U_d \backslash U$.
	
	Now let $z$ be a horizontal specialization of some $h \in U_d \backslash U$. We choose a valuation ring $(x_z,A)$ of $U$ which has $z$ as center. Using \cite[Lemma ~ 1.1.15 iv)]{Hu1993} we get $x_h \in U$ which is a horizontal generalization of $x_z$ and a vertical generalization of $h$. Let $(x_h,A')$ be a valuation ring with center $h$. By Lemma \ref{cmapprop} vi) we have $(x_h,A') \in T(U)$. We consider the same diagrams as in Lemma \ref{horizontal} and keep the same notation. From the proof of this lemma we know that $A \subset \pi^{-1}_z(\phi^{-1}(\pi_h(A')))$. We want to show that $(x_z,A) \in T(U)$. Assume the opposite and let $A \subset B \subset k(x_z)^+$ be a valuation ring such that $(x_z, B)$ has a center on $X$ but not on $U$. We denote the open affinoid neighbourhood of this center with $W$. Note that $W$ will play the role of $V$ in Lemma \ref{horizontal}. Since $B$ and $\pi^{-1}_z(\phi^{-1}(\pi_h(A')))$ contain a common valuation ring we can distinguish the following cases:
	
	First we assume $B \subset \pi^{-1}_z(\phi^{-1}(\pi_h(A')))$. By \cite[Remark ~ 1.1.14 e)]{Hu1996} there exists a valuation ring $B' \subset A'$ with $B = \pi^{-1}_z (\phi^{-1}(\pi_h(B')))$. We consider the second diagram in Lemma \ref{horizontal} where we have on the outer left side $\pi^{-1}_z (\phi^{-1}(\pi_h(B')))$ and on the outer right side $B'$ . Since $(x_z,B)$ has a center on $V$ we get first the left dashed arrow. By a similar argument as in this lemma, we get the right dashed arrow. This implies $(x_h,B')$ has a center on $V$ and thus $(x_h,A')$ has a center on $V$. The center of $(x_h,A')$ is not in $U$ because $h$ is not in $U$. This is a contradiction to $(x_h,A') \in T(U)$.
	
	Now we consider the case where $\pi^{-1}_z(\phi^{-1}(\pi_h(A'))) \subset B$. Again by \cite[Remark ~ 1.1.14 e)]{Hu1996} this time we get a valuation ring $B'$ which contains $A'$ and is contained in $k(x_h)^+$ such that $B = \pi^{-1}_z(\phi^{-1}(\pi_h(B')))$. Again, like in the other case, we get that $(x_h,B')$ has a center on $X$. This center is not on $U$, otherwise we argue again like in the Lemma \ref{horizontal} that $B$ will have a center on $U$. But this is a contradiction by the choice of $B$. In particular, this shows that $x_z$ is in $U$ because $(x_z,k(x_z)^+)$ has a center on $X$ and thus it is in $U$. This finish the proof of (v).
\end{proof}

Again, the following lemma is not new, but we have modified the proof by avoiding the use of the fact that 
$S(U)$ is pro-constructible. Instead, we use directly that $X$ is taut. 

\begin{lemma}\label{minimal}
	Let $X$ be a taut adic space. Then there exists a valuation ring $B$ of $k(x)$ such that $A \subset B \subset k(x)^+$ and $(x,B)$ has a center on $X$ and for every valuation ring $C$ of $k(x)$ with $A \subset C \subsetneq B$ the valuation ring $(x,C)$ of $X$ has no center on $X$.
\end{lemma}

\begin{proof}
	Consider the family $(B_i | B_i \ \text{ is a valuation ring with support} \ x$  and has a center on  $X$  and  $A \subset B_i \subset k(x)^+ )_{i \in I}$. Note that this is non-empty, because of $(x,k(x)^+)$. Since all those valuation rings $B_i$ contain $A$, they are totally ordered by inclusion.
	It is obvious that if such a $B$ exists, it will be the intersection of all those valuation rings $B_i$. Hence we have to show that this intersection $B$ is a valuation ring in $k(x)$ and has a center on $X$. The former is clear since it contains the valuation ring $A$. We show that the centers $z_i$ of all such valuation rings $B_i$, which have a center on $X$, are contained already in the same affinoid neighbourhood. Then with Lemma \ref{criterion} i) this shows the existence of a center of $B$. Note that by Lemma \ref{val} i) all those specializations are also totally ordered. For that we will show the following statement and its negation both implies that such an open affinoid neighbourhood which contains all centers exist.
	We choose an affinoid neighbourhood $U_i$ of some $z_i$. Since $X$ is taut, the closure $\overline{U_i}$ is quasi-compact. Hence there is a finite affinoid cover $V_j$ of $\overline{U_i}$. Every center $z_j$, which generalizes $z_i$, is contained in $U_i$ and every center $z_j$, which is a specialization of $z_i$, is contained in the closure of $U_i$. We consider the following statement: There is a $V_j$ such that for all centers $z_k$ there is a $z_j \in V_j$ with $z_j$ is a specialization of $z_k.$ If this statement is true, then all centers $z_k$ are contained in $V_j$. This is obviously true, because every center $z_k$ is a generalization of an element in $V_j$. Thus, it holds that $z_k \in V_j$. Now we consider the negation of the statement before. We assume that for all $V_j$ there is a center $z_k$ such that for all $z_j \in V_j$ we get that $z_k$ is a specialization of $z_j$. This means for every $V_j$ there is a smallest element with respect to the specialization order among all those centers of $B_i$  which are contained in this $V_j$. Since we only have finitely many $V_j$ we have even a smallest center $z_k$, which is contained in some $V_j$. Hence all centers must be in this $V_j$. This proves the lemma. 
\end{proof}

We need to make an additional assumption on 
$X$ and $S$ in order to successfully extend Huber's proof of the existence of a universal compactification. As mentioned in Remark \ref{newdef}, we could instead drop the conditions on $X$ and $S$ and require that $f$ is weakly square complete.

\begin{theorem}\label{main}
	Let $X$ be a weakly square complete adic space and $S$ a square complete and stable adic space. Then every morphism of adic spaces $f \colon X \rightarrow S$ which is locally of \hspace{0.05cm}$^+$weakly finite type, separated and taut has a  universal vertical compactification $(X', f', j)$ where $j \colon X \rightarrow X'$ is a quasi-compact open embedding, $j(X)$ is closed under generalized horizontal specializations, every point of $X'$ is a vertical specialization of a point of $j(X)$ and $\mathcal{O}_{X'} \rightarrow j_*\mathcal{O}_X$ is an isomorphism of sheaves of topological rings.
\end{theorem}

We give a brief overview of the proof. The main idea is to use Lemma \ref{mod} iv). At the beginning we will show why we can assume $S$ to be an affinoid space. Then for every open affinoid subspace $U$ and $V$ of $X$ we consider $U_d$ and $V_d$ from Definition \ref{cmap}. By Lemma \ref{openset} $U_d$, respectively $V_d$,  are open subsets of $U_c$, respectively of $V_c$. We will glue $U_d$ and $V_d$ along $(U\cap V)_d$. Our universal vertical compactification $f' \colon X' \rightarrow S$ and $j \colon X \hookrightarrow X'$ are morphism arising from the gluing. The proof is complete once we have verified that the hypotheses of the Lemma \ref{mod} iv) hold.
\begin{proof}
Let $\{S_i \ | \ i \in I\}$ be an affinoid open covering of $S.$ Assume that for every $i \in I$ there exists a vertical compactification $(X'_i, f'_i, j_i)$ of $f_i \colon f^{-1}(S_i) \rightarrow S_i$ such that $j_i$ is a quasi-compact open embedding, the image $j_i(f^{-1}(S_i))$ is closed under generalized horizontal specialization, every point of $X'_i$ is a vertical specialization of $j_i(f^{-1}(S_i))$ and $\mathcal{O}_{X'_i} \rightarrow {j_{i}}_*\mathcal{O}_{f^{-1}(S_i)}$ is an isomorphism of sheaves of topological rings. For every $i,j\in I$ we consider the open subspace $U_{ij} \coloneqq j_i(f^{-1}(S_i \cap S_j))$ in $X'_i$. We define $(j_{ij} \coloneqq {j_j \circ j^{-1}_i}|_{j_i(f^{-1}(S_i \cap S_j))} \ | \ i,j \in I)$, hence we get $j_{ij}$ is an isomorphism of adic spaces between $U_{ij}$ and $U_{ji}$. Together with $(X'_i \ | \ i\in I)$ this forms a gluing data and we get the glued space $X'$ with the morphism $j \colon X \rightarrow X'$ and $f' \colon X' \rightarrow S$ such that $f = f' \circ j$. The conditions of Lemma \ref{mod} iv) hold for $(X', f', j)$  and setting $S=U$ in this lemma shows, that $f'$ is indeed a universal vertical compactification of $f$. Here, we only explicitly show  that $j(X)$ is closed under generalized horizontal specializations. Let $z \in X'$ and $x \in X$ such that $z$ is a generalized horizontal specialization of $j(x)$. We can assume that for some $i,j \in I$ we have $z \in X'_i$ and $j(x) \in X'_j$. Thus, $j(x) \in X'_i \cap X'_j$. The intersection identifies with $j_i(f^{-1}(S_i \cap S_j))$ by definition of the gluing datum. In particular, $j(x) \in j^{-1}_i(f^{-1}(S_i))$ and that implies $j(x) =j_i(x)$. Since $j_i(f^{-1}(S_i))$ is closed under generalized horizontal specializations in $X'_i$, this shows $z \in j(X)$.  The rest follows immediately from the definition of the gluing data. Every affinoid space is weakly square complete. By Proposition \ref{weaksqmor} the subspace $f^{-1}(S_i)$ is weakly square complete.  Hence, we will assume for the rest of this proof that $S$ is affinoid.
\newline
Now we will construct $f' \colon X' \rightarrow S$ from a gluing datum.
By Lemma \ref{openset} ii) the subspace $U_d$ is open in $U_c$ for every $U \in \mathcal{F}$. We endow $U_d$ with the subspace topology of $U_c$. By Lemma \ref{openset} iii) the set $Q_{U,V}\coloneqq \psi_{U,U \cap V}((U \cap V)_d) \subset U_d$ is open in $U_d$ and $\psi_{U,U \cap V}:(U \cap V)_d \rightarrow Q_{U,V}$ is an isomorphism. We set $\lambda_{V,U} \coloneqq \psi_{V,U \cap V} \circ \psi^{-1}_{U,U \cap V}: Q_{U,V} \xrightarrow{\sim} Q_{V,U}.$ We show
\newline

(1) ($Q_{U,V}$ \ | \ $(U,V) \in \mathcal{F}^2)$ and $(\lambda_{V,U} \ | \ (U,V) \in \mathcal{F}^2)$ form a gluing datum for the family of adic spaces $(U_d \ | \ U \in \mathcal{F})$.
\newline

Let $U,V,W$ be elements of $\mathcal{F}$. Then $Q_{U,V} = c_U(T(U \cap V))$ and $Q_{U,W} = c_U(T(U \cap W)).$ It follows 

\begin{align*}
	Q_{U,V} \cap Q_{U,W} &= c_U(T(U \cap V)) \cap c_U(T(U \cap W)) \\
	&=  c_U(c^{-1}_U(c_U(T(U \cap V)))) \cap c_U(T(U \cap W))  \\
	&= c_U(c^{-1}_U(c_U(T(U \cap V))) \cap c_U^{-1}(c_U(T(U \cap W)))) \\
	&= c_U(T(U \cap V) \cap T(U \cap W))
\end{align*}
by applying Lemma \ref{cmapprop} vi) to the second and last equality and the third equality is by a general set theoretical identity for functions. Namely  $g(H) \cap T = g(H \cap g^{-1}(T))$ and setting $g= c_U$ and $H = c_U^{-1}(c_U(T( U \cap V)))$ and $T= c_U(T(U \cap W))$ explains the third equality above. With Lemma \ref{cmapprop} iv) we obtain $Q_{U,V} \cap Q_{U,W} =c_U(T(U \cap V \cap W)).$ Then $(\lambda_{V,U}(Q_{U,V} \cap Q_{U,W}) = c_V(T(U \cap V \cap W)) = Q_{V,U} \cap Q_{V,W}.$  Applying Proposition \ref{relcl} with Lemma \ref{mod} i) to $U\cap V \cap W \rightarrow S$ we get $(\lambda_{W,U} |_{(Q_{U,V} \cap Q_{U,W})}) = (\lambda_{W,V} |_{(Q_{V,W} \cap Q_{V,U})}) \circ (\lambda_{V,U} |_{(Q_{U,V} \cap Q_{U,W})}). $ This shows (1).
\\

Let $X'$ be the adic space which is given by the gluing of the $U_d(U \in \mathcal{F})$ according the gluing datum (1). Every $U_d$ is an adic space over $S$ since $U_c$ is already an adic space over $S$ and every $\lambda_{V,U}$ is an $S$-morphism. Hence $X'$ has a natural structure morphism $f' \colon X' \rightarrow S.$ By Lemma \ref{openset} i) we have a natural $S$-morphism $j \colon X \rightarrow X'$. It is left to show that Lemma \ref{mod} iv) hold for $j, f$ and $f'$, that means we have to show:
\\

 i) $j$ is a quasi-compact open embedding whose image is closed under generalized horizontal specialization image such that every point of $X'$ is a specialization of a point of $j(X)$ and $\mathcal{O}_{X'} \rightarrow j_*\mathcal{O}_X$ is an isomorphism of sheaves of topological rings,
\\

ii) $f'$ is locally of $^+$weakly finite type,
\\

iii) $f'$ is separated,
\\

iv) $f'$ is universally vertically specializing.
\\

We consider every $U_d$ as an open subspace of $X'.$ Then $U_d \cap V_d = (U \cap V)_d$ by definition of $X'$. 

i) First we show that $j$ is quasi-compact. Let $U \in \mathcal{F}$ be given. By Lemma \ref{openset} iv) we have for every $V \in \mathcal{F}$ that $V \cap Q_{V,U} \subset U \cap V$ which says $j^{-1}(U_d) \cap V \subset U \cap V$. This shows that $U = j^{-1}(U_d)$. We know that $j$ is injective by construction. Let $V \subset X'$ be an open quasi-compact subset.  In particular, we can assume it is covered by open affinoid subsets where each is contained in one of those $U_d.$ Since $V$ is quasi-compact, we can even assume $X'$ is covered by finitely many open affinoid subspaces where each is contained in one of those $U_d$. Together with $U=j^{-1}(U_d)$ and the fact that the inclusion $j|_U \colon U \hookrightarrow U_d$ is quasi-compact (because $\phi_U \colon U \rightarrow U_c$ is quasi-compact), this shows $j$ is quasi-compact.  Let $k:U_d \rightarrow U_c$ be the open inclusion. The composition $U \xrightarrow{j} U_d \xrightarrow{k} U_c$ is the morphism $\phi_U \colon U \rightarrow U_c$. Again together with $U = j^{-1}(U)_d$ and by Proposition \ref{relcl} iii) this shows $\mathcal{O}_{X'} \rightarrow j_*\mathcal{O}_X$ is an isomorphism of sheaves of topological rings. Further, let $z \in X'$ a generalized horizontal specialization of some $j(x)$ with $x \in X$. Choose some $U \in \mathcal{F}$ with $z \in U_d$. Then $j(x) \in U_d$. Since $U = j^{-1}(U_d)$ we get $j(x) \in j(U)$. Since $U_d \subset U_c$ and $(k \circ j(U))$ is closed under generalized horizontal specialization in $U_c$ by Lemma \ref{gen.hor.}, we get $z \in j(U)$.
\\

ii) In Theorem \ref{baby} we showed that $U_c \rightarrow S$ is of $^+$weakly finite type. Hence $U_d \rightarrow S$ is locally of $^+$weakly finite type and this shows ii).
\\

iii) Rather than proving that $f'$ is quasi-separated, we show that $X'$ is quasi-separated. Note that $X$ is quasi-separated because $S$ is affinoid by assumption and $f$ is separated.  Let $A$ and $B$ be quasi-compact open subsets of $X'$. We have to show that $A \cap B$ is quasi-compact. First we assume that there exist $U, V \in \mathcal{F}$ with $A \subset U_d$ and $B \subset V_d$. Let $\mu_{U,U \cap V} \colon ( U \cap V)_d \rightarrow U_d$ and $\mu_{V,U \cap V}:(U \cap V)_d \rightarrow V_d$ be the restrictions of $\psi_{U,U \cap V}:(U \cap V)_c \rightarrow U_c$ and $\psi_{V,U \cap V}:(U \cap V)_c \rightarrow V_c$ on $(U \cap V)_d$.
We show
\\

(2) $(U\cap V)_d = \psi^{-1}_{U,U \cap V}(U_d) \cap \psi^{-1}_{V, U \cap V}(V_d)$.
\\

We have $\psi^{-1}_{U,U \cap V}(U_d) = c_{U\cap V}(R(U \cap V) \cap T(U))$. The inclusion from left to right follows by Lemma \ref{cmapprop} i, v). The other inclusion follows directly by the definition of $c_{U\cap V}(R(U \cap V) \cap T(U))$ . Similarly we get $\psi^{-1}_{V,U \cap V}(V_d) = c_{U \cap V}(R(U \cap V) \cap T(V)).$ We obtain $\psi^{-1}_{U,U \cap V}(U_d) \cap \psi^{-1}_{V,U \cap V}(V_d) = c_{U \cap V}(R(U \cap V) \cap T(U) \cap T(V)) = c_{U \cap V} (T(U \cap V)) = (U \cap V)_d$, where the first equality follows by Lemma \ref{cmapprop} vii) by a similar argument as in (1), where we showed $Q_{U,V} \cap Q_{U,W} = c_U(T(U \cap V) \cap T((U \cap W))$, and the second equality follows by Lemma \ref{cmapprop} iv). Since $\psi_{U,U \cap V}$ and $\psi_{V,U \cap V}$ are morphisms between affinoid spaces induced by an adic ring homomorphism, they are quasi-compact. Hence $\psi^{-1}_{U,U \cap V}(A)$ and $\psi^{-1}_{V,U \cap V}$(B) are  quasi-compact open subsets of $(U \cap V)_c$. Then their intersection is also quasi-compact, because $(U \cap V)_c$ is a spectral space. By (2) we get $\psi^{-1}_{U,U \cap V}(A) \cap \psi^{-1}_{V,U \cap V}(B) = \mu^{-1}_{U,U \cap V}(A) \cap \mu^{-1}_{V,U \cap V}(B) = A \cap B,$ where the last equality follows from the identification of $U_d \cap V_d$ with $(U \cap V)_d$. Thus $A \cap B$ is quasi-compact. For the general case we cover $A$ with finitely many $U_d$. Since $A \cap U_d$ is open,  it is a union of affinoid open subspaces, which are all contained in $U_d$. Altogether we get that $A$ is the union of affinoid open subspaces, where each affinoid open subspace is completely contained in one $U_d$ and this union can be reduced to a finite union since $A$ is quasi-compact. We argue similarly for $B$. Then $A\cap B$ is just a finite union of sets which are quasi-compact by the special case. Hence $A \cap B$ is quasi-compact. 
\\

It remains to show for (iii) that $f'$ satisfies the valuative criterion for separatedness \cite[Proposition ~ 3.11.12.]{Hu1993}. Let $(x,A,A_{\nu})$ be a valuated valuation ring of $X'$. Since $S$ is an affinoid space by assumption, $(x,A,A_{\nu})$ has at most one center on $S$. Thus it is enough to show if $u,v \in X'$ are centers of $(x,A, A_{\nu})$ then $u = v$.  We choose $U,V \in \mathcal{F}$ with $u \in U_d$ and $v \in V_d$ and hence $x \in (U \cap V)_d$. We distinguish the cases $x \in U \cap V$ and $x \in (U \cap V)_d \backslash (U\cap V)$.
\\
If $x \in  (U \cap V)_d \backslash (U\cap V)$ then by Lemma \ref{openset} v) we get $u,v \in  (U \cap V)_d \backslash (U\cap V)$ and hence $u,v$ are centers of $(x,A,A_{\nu})$ on $(U \cap V)_d$. Therefore $u=v$, because $(U \cap V)_d$ is separated over $(U \cap V)_c$ by Lemma \ref{openset} ii). 

We assume $x \in U\cap V$. Since $(x,A,A_{\nu})$ has centers on $U_d$ and $V_d$, by Lemma \ref{criterion} iv) also $(x,A_{\nu})$ has centers on $U_d$ and $V_d$. Note that we can use this lemma, because $x \in U_c$ is in an affinoid space which contains the open subspace $U_d$, respectively $x \in V_c$ which contains the open subspace $V_d$ and thus both open subspaces contain a center of $(x,A_{\nu})$. Hence $(x,A_{\nu})$ is a valuation ring on $U$ and $V$ and $c_U((x,A_{\nu})) \in U_d$ and $c_V((x,A_{\nu})) \in V_d$. By Lemma \ref{cmapprop} v) we get $(x,A_{\nu}) \in T(U) \cap T(V)$ and by Lemma \ref{cmapprop} iv) we get $(x,A_{\nu}) \in T(U\cap V)$. This means $(x,A_{\nu})$ has a center $h \in (U \cap V)_d$. Thus all three centers in $U_d, V_d$ and $(U \cap V)_d$ of $(x,A_{\nu})$ are equal to $h$. Note that by Lemma \ref{criterion} iv) $u$ and $v$ are horizontal specializations of $h$. We distinguish whether $h \in (U \cap V)_d \backslash U \cap V$ or $h \in U \cap V$: If the first holds then by Lemma \ref{openset} v), $u$ and $v$ are in $(U \cap V)_d \backslash U \cap V.$ Like in the first case for $x$ we get $u=v$. If $h \in U\cap V$, then in particular $h \in X$ and thus $u,v$ are also in $X$. Since by (5.i) $X$ is closed under generalized horizontal specializations and $X \rightarrow S$ is separated and $S$ affinoid, we get $u=v$.
\\

iv) We show for $x \in X'$ that $f'$ is universally vertically specializing at $x$. By Lemma \ref{mod} iii) it is enough to show for a vertical generalization of $x$ that $f'$ is universally vertically specializing. So we can assume $x \in X$, because every point of $X'$ is a vertical specialization of a point of $X$. Let $(x,A)$ be a valuation ring of $X$ which has a center on $S$. We want to show that $(x,A)$ has a center on $X'$. Let us choose a $U \in \mathcal{F}$ with $x \in U$. 

Let $(x,B)$ be the valuation ring such that $A \subset B \subset k(x)^+$ and $(x,B)$ has a center on $X$ and for every valuation ring $C$ of $k(x)$ with $A \subset C \subsetneq B$ the valuation ring $(x,C)$ of $X$ has no center on $X$ which exist by Lemma \ref{minimal}. Let $y \in X$ be the center of $(x,B)$ on $X$ and let $i \colon k(y) \rightarrow k(x)$ be the natural ring homomorphism. Note that $k(y)^+ = i^{-1}(B).$ Set $D\coloneqq i^{-1}(A).$ For every valuation ring $C$ of $k(y)$ with $D \subset C \subsetneq k(y)^+$ the valuation ring $(y,C)$ of $X$ has no center on $X$, because otherwise we would have a valuation ring $(x,C')$ with a center on $X$, which is strict a subset of $B$. By the choice of $B$ this cannot be. Note that strictness will be conserved by taking intersections with $k(y)$ by similar arguments as in the proof of Lemma \ref{val} ii). We choose a $U \in \mathcal{F}$ with $y \in U$. By assumption $(x,A)$ has a center on $S$. Hence $(y,D)$ has a center on $S$ by the choice of $D$. By definition of $T(U)$ we get $(y,D) \in T(U)$. Thus there is a center $z \in U_d$ of $(y,D)$ by definition of $U_d$ and $z$ is also a center of $(x,A)$. This concludes the proof of this theorem. 
\end{proof}

\begin{corollary}(\cite[Theorem ~ 5.1.5]{Hu1996}) \label{original}
	Let $f \colon X \rightarrow S$ be a morphism of analytic adic spaces which is locally of $^+$weakly finite type, separated and taut and where $S$ is a stable adic space. Then $f$ has a universal compactification in the sense of analytic adic spaces, \cite[Definition ~ 5.1.1]{Hu1996}.
\end{corollary}

\begin{proof}

Every analytic adic space is square complete because there are no proper horizontal specializations. We also mention that the original proof of \cite[Theorem ~ 5.1.5]{Hu1996} is incorporated into our arguments in Theorem \ref{main}. 
\end{proof}

We prove that a certain category of analytic adic spaces admits a universal compactification for morphisms as in the corollary above

\begin{definition}
	Let $\mathcal{C}$ be the category of quasi-separated adic spaces over $\mathbb{Z}_p$ having an open covering by affinoids attached to stably strongly uniform Tate Huber Pairs.
\end{definition}

\begin{proposition}
	Let $f \colon X \rightarrow S$ be a morphism in $\mathcal{C}$ that is locally of $^+$weakly finite type, separated and taut. Then there exists a universal compactification of $f$ in $\mathcal{C}$. 
\end{proposition}

\begin{proof}
Note that $X$ and $S$ are analytic adic spaces. Together with the overall assumptions on $f$ and using Corollary \ref{original}  there exists a universal compactification $X'$ of $f$. We have to show that $X'$ is in $\mathcal{C}$. The universal compactification $X'$ is glued from open subspaces of $U_c$ for open affinoid subsets $U \subset X$. The proof is complete, if we show:
\begin{enumerate}[label=\roman*)]
	\item If $U\rightarrow S'$ is a morphism between affinoid spaces in $\mathcal{C}$, then $U_c$ is an object of the category $\mathcal{C}$.
	\item For any open subspace $U$ of some object $Y$ in $\mathcal{C}$ the subspace $U$ is also an object in $\mathcal{C}$.
\end{enumerate}
i) The ring of integral elements $I(U)$ of the Huber pair attached to $U_c$ is contained in the ring of integral elements that is attached to $U$. Thus $U_c$ is an object of $\mathcal{C}$. 
\\
ii) Let $Y$ be an object in $\mathcal{C}$ and $U \subset Y$ an open subspace. Let $U_i$ be an open affinoid cover of $Y$, where each $U_i$ is attached to stably strongly uniform Tate Huber pairs. We cover $U \cap U_i$ with rational opens of $U_i$, which are objects in $\mathcal{C}$. This shows ii).

\end{proof}

\section{Ad hoc construction} \label{4}

We want to describe the universal vertical compactification more directly. For this, we will construct a set $\mathrm{Spa}(X,S)$ for any morphism $f \colon X \rightarrow S$ of adic spaces, Definition \ref{square}, which will be bijective to the universal vertical compactification. We will first show that this holds in the affinoid case in Proposition \ref{bab}. In the affinoid case we will endow $\mathrm{Spa}(X,S)$ with the structure of an adic space by structure transport. For arbitrary adic spaces $X$ and $S$ we will define a final topology by the help of the affinoid case in Definition \ref{top}. In Theorem \ref{spa} we will show that for a morphism $f \colon X \rightarrow S$ with the same conditions as in Theorem \ref{main} the set $\mathrm{Spa}(X,S)$ will be even homeomorphic to $X'$. By this we can endow $\mathrm{Spa}(X,S)$ also with an adic structure.

\begin{definition}\label{square}
	
	Let $f \colon X \rightarrow S$ be a morphism of adic spaces. A $point$ $of$ $X$ $with$ $center$ $on$ $S$ is a commutative square of the form
	
	\begin{center}
		\begin{tikzcd}
			
			\text{Spa}(k,k^{\circ}) \arrow[r,"u"] \arrow[d] & X \arrow[d] \\
			\text{Spa}(k,k^+) \arrow[r, "v"] & S,\\

		\end{tikzcd}
	\end{center}
	
	where $u$ is an adic morphism, $(k,k^+)$ is an affinoid field with ring of powerbounded elements $k^{\circ}$ and where the left vertical map is induced by the inclusion $k^+ \subset k^{\circ}$. 
	We denote a point of $X$ with center on $S$ with $(k^+,u,v)$. Two points $(k_i^+, u_i, v_i)$ for $i=1,2$ of $X$ with center on $S$ are equivalent if there exists a third point $(V,u,v)$ and commutative diagrams
	
	\begin{center}
		\begin{tikzcd}
			\text{Spa}(k_i,k_i^{\circ}) \arrow[d] \arrow[r] \arrow[rr, bend left, "u_i"] & \text{Spa}(k,k^{\circ}) \arrow[r, "u"] \arrow[d] & X \arrow[d] \\
			\text{Spa}(k_i,k_i^+) \arrow[r] \arrow[rr, bend right, "v_i"] & \text{Spa}(k,V) \arrow[r, "v"] & S 			
		\end{tikzcd}
	\end{center}
	with $k_i^+ \cap k = V$ for $i=1,2$, where $V \subset k^{\circ}$ is a valuation ring of $k$.
\end{definition}

\begin{lemma}\label{triple}
	Let $X \rightarrow S$ be a morphism of adic spaces.
	For every point of $X$ with center on $S$ exists a unique triple $(x,V,s)$ such that $(x,V)$ is a valuation ring on $X$ and $x$ is a point without vertical generalization and $s$ is center of $(x,V)$ on $S$ such that we have a unique factorization
	
	\begin{center}
		\begin{tikzcd}
			\mathrm{Spa}(k,k^{\circ}) \arrow[d] \arrow[r] \arrow[rr, bend left, "u"] & \mathrm{Spa}(k(x),k(x)^{\circ}) \arrow[r, "u_x"] \arrow[d] & X \arrow[d] \\
			\mathrm{Spa}(k,k^+) \arrow[r] \arrow[rr, bend right, "v"] & \mathrm{Spa}(k(x),V) \arrow[r, "v_x"] & S.			
		\end{tikzcd}
	\end{center}
	and $(k(x), u_x, v_x)$ is equivalent to $(k,u,v)$.
\end{lemma}

\begin{proof}

We consider $x \in X$ which is given by the image of Spa$(k,k^{\circ}) \rightarrow X$. By assumption $u$ is an adic morphism. This implies $x$ has no vertical generalization since it has height 0 or 1 depending on $(k,k^{\circ})$. We choose $s \in S$ as the image of the closed point of Spa$(k,k^+)$ under Spa$(k,k^+) \rightarrow S$. We have that $f(x)$ is a vertical generalization of $s$. Then we set $V \coloneqq k(x) \cap k^+$ and $u_x \colon \mathrm{Spa}(k(x),k(x)^{\circ}) \rightarrow X$ is the natural map. Note that $V \subset k(x)^+$ because $k(x)^+ =k(x)^{\circ}$ and thus $V$ is a valuation ring of $x$. We have that $s$ is a center of $(x,V)$ on $S$. Hence, we get a natural map $v_x \colon \mathrm{Spa}(k(x),V) \rightarrow S$ and $V$ has to be chosen like that for $(k(x),u_x,v_x)$ to be equivalent to $(k,u,v)$.

\end{proof}

\begin{lemma}
	Let $X \rightarrow S$ be a morphism of adic spaces. The defined relation on points of $X$ with center on $S$ in Definition \ref{square} is an equivalence relation. 
\end{lemma}

\begin{proof}
By definition the relation is reflexive and symmetric.
We consider three points $(k_i,u_i,v_i)$ of $X$ with center on $S$ for $i=1,2,3$ such that $(k_1,u_1,v_1)$ is equivalent to $(k_2,u_2,v_2)$ and the latter triple is equivalent to $(k_3,u_3,v_3)$. Let $(k,u,v)$ and $(k',u',v')$ be witnesses that the triples $(k_i,u_i,v_i)$ are equivalent for $i=1,2$, and $i=2,3$ respectively. We denote the valuation rings of $k$ and $k'$ from the second diagram in the bottom row with $V$ and $V'$. We get for $(k,u,v)$ and $(k',u',v')$ by the lemma above unique triples $(x,W,s)$ and $(x',W',s')$. Since $(k_2,u_2,v_2)$ is equivalent to $(k,u,v)$ and $(k',u',v')$ we get that $(x,W,s)$ and $(x,',W',s')$ are also triples for $(k_2,u_2,v_2)$. By uniqueness of this triple we get $x=x', W=W'$ and $s=s'$.  Furthermore we have $k_3^+ \cap k(x) = k_3^+ \cap k' \cap k(x) = V' \cap k(x) = W = V \cap k(x) = k_1^+ \cap k \cap k(x) = k_1^+ \cap k(x)$. Thus, we get the following commutative diagram
\begin{center}
	\begin{tikzcd}
		\mathrm{Spa}(k_i,k_i^{\circ}) \arrow[d] \arrow[r] \arrow[rr, bend left, "u_i"] & \mathrm{Spa}(k(x), k(x)^{\circ}) \arrow[d] \arrow[r] & X \arrow[d] \\
		\mathrm{Spa}(k_i,k_i^+) \arrow[r] \arrow[rr, bend right, "v_i"] & \mathrm{Spa}(k(x), W) \arrow[r] & S, \\
	\end{tikzcd}
\end{center}
for $i=1,3$.

\end{proof}

\begin{definition}\label{expl}
	
	For a morphism $X \rightarrow S$ of adic spaces, we denote by $\mathrm{Spa}(X,S)$  the set of equivalence classes of points of $X$ with center on $S$.
	
\end{definition}

\begin{lemma}
	
	Let $f \colon X \rightarrow S$ be a morphism of adic spaces and $U \subset X$ and $V \subset S$ open affinoid subsets with $f(U) \subset V$. Then there exists a injective map $\mathrm{Spa}(U,V) \hookrightarrow \mathrm{Spa}(X,S)$.
	
\end{lemma}

\begin{proof}

This is a consequence of the definition of $\mathrm{Spa}(U,V)$ 
and $\mathrm{Spa}(X,S)$.

\end{proof}

\begin{proposition}\label{bab}
	
	For a morphism $f \colon \mathrm{Spa}(A,A^+) \rightarrow \mathrm{Spa}(R,R^+)$ the set $\mathrm{Spa}(U,S)$ is bijective to $U_c$.
	
\end{proposition}

\begin{proof}
We denote $U \coloneqq  \mathrm{Spa}(A,A^+)$ and $S \coloneqq \mathrm{Spa}(R,R^+)$.
By Lemma \ref{triple} the set Spa$(U,S)$ consists of triples $(x,V,s)$ with $(x,V)$ a valuation ring on $U$ with $s$ a center on $S$ of this valuation ring and such that $x$ is maximal with respect to vertical generalizations. Then in Remark \ref{lift} we mentioned that we can lift the center $s$ uniquely to a center of $(x,V)$ on $U_c$. This gives a well-defined map $\phi \colon $Spa$(U,S) \rightarrow U_c$. Additionally in Proposition \ref{relcl} ii) we saw that  every point $x'$ of $U_c$ is a vertical specialization of a point $x$ in $U$ which has rank 0 or rank 1, that means $x$ has no vertical generalizations. In particular, there is a unique triple $(x,V,f(x'))$ that is mapped to $x'$ under $\phi$, establishing the bijectivity of this map.

\end{proof}

We endow, by the lemma above, for a morphism $U \rightarrow S$ of affinoid adic spaces the set Spa$(U,S)$ with an adic structure, which is given by $U_c$. 
This gives rise to the following definition.
\begin{definition}\label{top}
	Let $X \rightarrow S$ be a morphism of adic spaces and $U \subset X, V \subset S$ open affinoid subsets such that the image of $U$ is contained in $V$ with $U \rightarrow V$ arbitrary. We have the inclusion Spa$(U,V) \hookrightarrow$ Spa$(X,S)$. Then we endow Spa$(X,S)$ with the final topology for the family of all such inclusions.
\end{definition}

\begin{theorem}\label{spa}
	
	Let $X$ be a weakly square complete adic space and $S$ a square complete and stable adic space and $f \colon X \rightarrow S$ a morphism of locally of $^+$weakly finite type, separated and taut and $f' \colon X' \rightarrow S$ the universal vertical compactification of $f$. Then $\mathrm{Spa}(X,S)$ is homeomorphic to $X'$.
	
\end{theorem}

\begin{proof}

In Theorem \ref{main} we saw that every point of $X'$ is a vertical specialization of a point of $X$ which has no vertical generalization. Then the proof works exactly the same as in Lemma \ref{bab} and we get a bijection $\phi \colon $Spa$(X,S) \rightarrow X'$, where we map $(x,V,s)$ to $x' \in X'$ which is a center of $(x,V)$ with $f'(x') = s$. Note that this assignment is well-defined, because $f'$ is vertically separated.
\bigskip

We show that $\phi$ is continuous. Let $W \subset X'$ be an open affinoid subset. We have to show that for open affinoid subsets $U \subset X$, $V \subset S$ with $f(U) \subset V$ we get that $\psi^{-1}(\phi^{-1}(W))$ is open, where $\psi \colon$ Spa$(U,V) \hookrightarrow$ Spa$(X,S)$ is the inclusion. Since the adic structure on Spa$(U,V)$ comes from $\overline{U}^V$, the bijection between them is actually an isomorphism of adic spaces. We have to show that the following diagram commutes
\begin{center}
	\begin{tikzcd}
		\text{Spa}(X,S) \arrow[r, "\phi"] & X' \\
		\text{Spa}(U,V) \arrow[u, hookrightarrow, "\psi"] \arrow[r, "\cong"] &\overline{U}^V  \arrow[u, "h", swap]  \\
	\end{tikzcd}
\end{center}

and that the right-hand side vertical arrow exists as morphism of adic spaces. For this we consider the diagram below

\begin{center}
	\begin{tikzcd}
		U \arrow[r, hookrightarrow] \arrow[d] & f^{-1}(V)  \arrow[r] \ \arrow[d] & V \arrow[r, hookrightarrow] & S \\
		\overline{U}^V \arrow[r, dashed] \arrow[urr, "g", near start, crossing over] \arrow[rr,bend right, "h"] & f'^{-1}(V) \arrow[ur] \arrow[r, hookrightarrow] & X'. \arrow[ur, "f'"] \\
	\end{tikzcd}
	
\end{center}
Note that by Theorem \ref{main} the conditions of Lemma \ref{mod} iv) hold. Hence $f'^{-1}(V) \rightarrow V$ is a universal vertical compactification of $f^{-1}(V) \rightarrow V$, in particular the former map is vertically partially proper. Since $\overline{U}^V \rightarrow V$ is a universal vertical compactification of $U \rightarrow V$, this explains the existence of the dashed arrow and also the right vertical arrow of the first diagram above. Clearly the second diagram is commutative. Now we explain why the first diagram is commutative. Consider a point $(x,A,s)$ in Spa$(U,V)$ with $x \in U$ and $s \in V$. This point is mapped to a point $x' \in \overline{U}^V$ such that $g(x') = s$ and $x'$ is a center of $(x,A)$ and by commutativity of the second diagram it follows that $f'(h(x')) = g(x') =s$. Now $\psi$ is just the inclusion and $\phi$ maps $(x,A,s)$ to a point $x'' \in X'$ with $f'(x'') =s$ and $x''$ is a center of $(x,A)$. Since $f'$ is vertically separated we get $h(x') = x''$. This shows the commutativity of the first diagram above and the continuity of $\phi$, because $h$ is continuous.
\bigskip

It is left to show that $\phi$ is open. Let $W$ be an open subset of Spa$(X,S)$, this means for every $\psi$ as defined before $\psi^{-1}(U)$ is open. Together with $\phi$ being a bijection and the first diagram above being commutative, we get that $h^{-1}(\phi(W))$ is open. Since $V$ is an affinoid set and $f'^{-1}(V)$ is a universal vertical compactification of $f^{-1}(V) \rightarrow V$, we know from the proof of Theorem \ref{main} that $f'^{-1}(V)$ is covered by sets of the form $U_d$ for open affinoid subsets $U \subset f^{-1}(V)$. The following diagram
\begin{center}
	\begin{tikzcd}
		U \arrow[r, hookrightarrow]  \arrow[d, hookrightarrow] & f^{-1}(V) \arrow[dd, "j"] & \\
		U_d \arrow[dr, "\iota"] \arrow[d, hookrightarrow] & & \\
		\overline{U}^V \arrow[r, "g'"] & f'^{-1}(V) \arrow[r] & V \\		
	\end{tikzcd}
\end{center}
is commutative: Because if $x' \in U_d \subset \overline{U}^V$ is a center of a valuation ring $(x,A)$ with support in $U$ we get that $\iota(x')$, $g'(x')$ are centers of $(x,A)$ on $f'^{-1}(V)$. Together with $f'^{-1}(V) \rightarrow V$ being vertically separated and $V$ being an affinoid space this implies that $(x,A)$ has at most one center on $f'^{-1}(V)$. Thus we have $\iota(x') = g'(x')$. Let denote $U_d \rightarrow f'^{-1}(V) \hookrightarrow X'$ with $\iota'$. Since $h^{-1}(\phi(W))$ is open we get by the commutative diagram above that $\iota'^{-1}(\phi(W))$ is open. This implies that $\phi(W) \cap U_d \subset f'^{-1}(V)$ is open. Consequently $f'^{-1}(V) \cap \phi(W)$ is open because if all $U$ cover $f^{-1}(V)$ we have all $U_d$ cover $f'^{-1}(V)$. Finally $\phi(W)$ is open in $X'$, since $f'^{-1}(V)$ is an open covering of $X'$.
\end{proof}

With the last theorem, we can view $\mathrm{Spa}(X,S)$ as an adic space via structure transport. We would like to close the article with the following remark:

\begin{remark}    
	
	As a final remark, we would like to mention that there is an open conjecture about whether one can endow $\mathrm{Spa}(X,S)$ with the structure of an adic space under assumptions weaker than those in Theorem \ref{main}. In particular, it would be interesting to achieve this for arbitrary $X$.
\end{remark}

\bibliographystyle{alpha}
\bibliography{quellen}

\end{document}